\documentclass{amsart}

\usepackage{amsmath,amsfonts,amsthm,amssymb,tikz,enumerate}
\usetikzlibrary{calc}
\usepackage[margin=1in]{geometry}

\DeclareMathOperator{\Assh}{Assh}
\newcommand{\BC}[2]{\left(\!\!\!\begin{array}{c} #1\\ #2  \end{array}\!\!\right)}
\DeclareMathOperator{\cone}{cone}
\DeclareMathOperator{\Exp}{Exp}
\newcommand{\ii}[2]{I^{\lceil #1 #2\rceil}+J^{[#2]}}
\newcommand{\iii}[4]{{#3}^{\lceil #1 #2\rceil}+{#4}^{[#2]}}
\newcommand{\fthresh}[2]{c_{#2}\left({#1}\right)}
\DeclareMathOperator{\Hull}{Hull}
\newcommand{\length}[1]{\lambda\!\left(#1\right)}
\newcommand{\lengthR}[2]{\lambda_{#1}\!\left(#2\right)}
\newcommand{\m}{\mathfrak{m}}
\newcommand{\n}{\mathfrak{n}}
\newcommand{\N}{\mathbb{N}}
\newcommand{\nfact}[2]{\mathcal{H}_{#2}\!\left({#1}\right)}
\newcommand{\p}{\mathfrak{p}}
\newcommand{\q}{\mathfrak{q}}
\newcommand{\R}{\mathbb{R}}
\DeclareMathOperator{\rank}{rank}
\newcommand{\rmu}[3]{h_{#1}\!\left( #2,#3\right)}
\newcommand{\rmuM}[4]{h_{#1}\!\left( #2,#3;#4\right)}
\newcommand{\rmuR}[4]{h^{#1}_{#2}\!\left( #3,#4\right)}
\newcommand{\rmuu}[2]{h_{#1}\!\left({#2}\right)}
\newcommand{\rmuRM}[5]{h^{#1}_{#2}\!\left( #3,#4;#5\right)}
\newcommand{\rmuuM}[3]{h_{#1}\!\left({#2};{#3}\right)}
\newcommand{\scl}[2]{{#2}^{\mathrm{cl}_{#1}}}
\newcommand{\set}[2]{\left\{#1\;|\;#2\right\}}
\newcommand{\smu}[3]{e_{#1}\!\left({#2},{#3}\right)}
\newcommand{\smuM}[4]{e_{#1}\!\left({#2},{#3};{#4}\right)}
\newcommand{\smuR}[4]{e^{#1}_{#2}\!\left({#3},{#4}\right)}
\newcommand{\smuRM}[5]{e^{#1}_{#2}\!\left({#3},{#4};{#5}\right)}
\newcommand{\smuu}[2]{e_{#1}\!\left({#2}\right)}
\newcommand{\smuuM}[3]{e_{#1}\!\left({#2};{#3}\right)}
\newcommand{\smuuR}[3]{e^{#1}_{#2}\!\left({#3}\right)}
\DeclareMathOperator{\Spec}{Spec}
\newcommand{\standardassumptionsgreaterthan}{Let $(R,\m)$ be a local ring of characteristic $p>0$, let $I$ and $J$ be $\m$-primary ideals of $R$, let $M$ be a finitely generated $R$-module, and let $s> 0$.\ }
\newcommand{\standardassumptionsdimgreaterthan}{Let $(R,\m)$ be a local ring of dimension $d$ and characteristic $p>0$, let $I$ and $J$ be $\m$-primary ideals of $R$, let $M$ be a finitely generated $R$-module, and let $s> 0$.\ }
\newcommand{\standardassumptionsnos}{Let $(R,\m)$ be a local ring of characteristic $p>0$, let $I$ and $J$ be $\m$-primary ideals of $R$, and let $M$ be a finitely generated $R$-module.\ }
\newcommand{\standardassumptionsnosdim}{Let $(R,\m)$ be a local ring of dimension $d$ and characteristic $p>0$, let $I$ and $J$ be $\m$-primary ideals of $R$, and let $M$ be a finitely generated $R$-module.\ }
\newcommand{\standardassumptionsnosnom}{Let $(R,\m)$ be a local ring of characteristic $p>0$ and let $I$ and $J$ be $\m$-primary ideals of $R$.\ }
\newcommand{\sthresh}[2]{b_{#2}\left({#1}\right)}
\DeclareMathOperator{\vol}{vol}
\newcommand{\wsc}[2]{{#2}^{\mathrm{w.cl}_{#1}}}
\newcommand{\Z}{\mathbb{Z}}

\newtheorem{theorem}{Theorem}[section]
\newtheorem{corollary}[theorem]{Corollary}
\newtheorem{lemma}[theorem]{Lemma}
\newtheorem{proposition}[theorem]{Proposition}
\theoremstyle{definition}
\newtheorem{definition}[theorem]{Definition}
\newtheorem{remark}[theorem]{Remark}
\newtheorem{example}[theorem]{Example}

\begin{document}

\title{Interpolating Between Hilbert-Samuel and Hilbert-Kunz Multiplicity}
\author{William D. Taylor}
\address{Department of Mathematical Sciences \\ University of Arkansas \\ Fayetteville \\ AR 72701}
\email{wdtaylor@uark.edu}

\begin{abstract}We define a function, called $s$-multiplicity, that interpolates between Hilbert-Samuel multiplicity and Hilbert-Kunz multiplicity by comparing  powers of ideals to the Frobenius powers of ideals.  The function is continuous in $s$, and its value is equal to Hilbert-Samuel multiplicity for small values of $s$ and is equal to Hilbert-Kunz multiplicity for large values of $s$.  We prove that it has an Associativity Formula generalizing the Associativity Formulas for Hilbert-Samuel and Hilbert-Kunz multiplicity.  We also define a family of closures such that if two ideals have the same $s$-closure then they have the same $s$-multiplicity, and the converse holds under mild conditions.  We describe the $s$-multiplicity of monomial ideals in toric rings as a certain volume in real space.\end{abstract}

\maketitle

\section{Introduction}
The purpose of this paper is to investigate a function that interpolates continuously between Hilbert-Samuel multiplicity and Hilbert-Kunz multiplicity.  First we define a limit that behaves like a multiplicity, then we normalize it to get a proper interpolation between the Hilbert-Samuel and Hilbert-Kunz multiplicities.  This interpolating function, which we call $s$-multiplicity, is a single object which captures the behavior of both multiplicities as well as a family of multiplicity-like functions between them.  Many of the similarities between the two multiplicities, such as the existence of an Associativity Formula and the connection to a closure, can be interpreted as special cases of a more general statement about $s$-multiplicity.  

Throughout this paper, all rings will be assumed noetherian.  By $\lengthR{R}{M}$ we mean the length of $M$ as an $R$-module.  When the ring $R$ is understood we may write $\length{M}$ for $\lengthR{R}{M}$.

\begin{definition} Let $(R,\m)$ be a local ring of dimension $d$, $I\subseteq R$ an $\m$-primary ideal of $R$, and $M$ a finitely generated $R$-module.  The \emph{Hilbert-Samuel multiplicity of $M$ with respect to $I$} is defined to be
\[e(I;M)=\lim_{n\to\infty}\frac{d!\cdot \length{M/I^nM}}{n^d}.\]
We often write $e(I)$ for $e(I;R)$.
\end{definition}

Many properties of the Hilbert-Samuel multiplicity are well known.  For instance, if $I\subseteq J$ are ideals that have the same integral closure, then $e(I)=e(J)$, and if $R$ is formally equidimensional, then the converse holds \cite{Rees-aTransformsLocalRings}.  The Hilbert-Samuel multiplicity is always a positive integer, $e(\m)=1$ if $(R,\m)$ is regular, and if $R$ is formally equidimensional the converse holds \cite[Theorem 40.6]{Nagata-LocalRings}.

When $R$ is of prime characteristic $p>0$, the Frobenius map $F:R\to R$ taking $r\mapsto r^p$ is a ring homomorphism, and so we may treat $R$ as a module over itself via the action $r\cdot x=r^p x$.  In this case, we often denote the module $R$ with this new action by $F_\ast R$, and elements of this module by $F_\ast r$ for $r\in R$.  An $R$-module homomorphism $\varphi:F_\ast R\to R$ is called a $p^{-1}$-linear map, and has the property that for any $r,x\in R$, $r\varphi(F_\ast x)=\varphi(F_\ast (r^p x))$.  If $F_\ast R$ is finitely generated as an $R$-module, we say the ring  $R$ is $F$-finite.  For an ideal $I\subseteq R$ and $e\in \N$,  the $e$th Frobenius power of $I$, denoted $I^{[p^e]}$, is the ideal generated by the $p^e$th powers of the elements of $I$, equivalently by the $p^e$-th powers of a set of generators for $I$.  For any $p^{-1}$-linear map $\varphi$ and ideal $I\subseteq R$, $\varphi(F_\ast (I^{[p]}))\subseteq I$.

When $R$ is a ring of positive characteristic, we can define a limit similar to the Hilbert-Samuel multiplicity using the Frobenius powers of the ideal instead of the powers.

\begin{definition}  Let $(R,\m)$ be a local ring of dimension $d$, $I\subseteq R$ an $\m$-primary ideal of $R$, and $M$ a finitely generated $R$-module.  The \emph{Hilbert-Kunz multiplicity of $M$ with respect to $I$} is defined to be
\[e_{HK}(I;M)=\lim_{e\to\infty}\frac{\length{M/I^{[p^e]}M}}{p^{ed}}.\]
We often write $e_{HK}(I)$ for $e_{HK}(I;R)$.
\end{definition}

The Hilbert-Kunz multiplicity has some properties similar to the Hilbert-Samuel multiplicity.  In particular, if $I\subseteq J$ are ideals that have the same tight closure, then $e_{HK}(I)=e_{HK}(J)$, and if $R$ is complete and equidimensional then the converse holds \cite[Theorem 8.17]{HochsterHuneke-TightClosureInvariantTheory}.  The Hilbert-Kunz multiplicity is a real number at least 1, though unlike the Hilbert-Samuel multiplicity it need not be an integer.  However, like the Hilbert-Samuel multiplicity, $e_{HK}(\m)=1$ if $(R,\m)$ is regular, and if $R$ is unmixed then the converse holds \cite[Theorem 1.5]{WatanabeYoshida-HKMultiplicityInequality}.

A brief outline of the paper is as follows: In Section \ref{rmult section}, we prove the existence of the limit used to define the $s$-multiplicity and establish many of its fundamental properties that we will use throughout the paper.  Of particular note are the results that the $s$-multiplicity is continuous in the parameter $s$ and the Associativity Formula for $s$-multiplicity.  In Section \ref{smult section}, we examine the relationship between the $s$-multiplicity and the Hilbert-Samuel and Hilbert-Kunz multiplicity and compute the limit from Section \ref{rmult section} for regular rings, which allows us to finish the definition of the $s$-multiplicity. 
In Section \ref{closure operators section}, we define a collection of closures and prove that they have exactly the same relationship with $s$-multiplicity as integral closure (resp.\ tight closure) has with Hilbert-Samuel (resp.\ Hilbert-Kunz) multiplicity.  In Section \ref{smult toric rings section}, we describe a method for computing the $s$-multiplicity of pairs of ideals in toric rings and use it to compute the $s$-multiplicity of the $A_n$ singularities and rational normal curves.

\vspace{10pt}

\noindent \textbf{Acknowledgments.} The author would like to thank his Ph.D.\ advisor, Mark Johnson, for very many fruitful discussions and much excellent advice.  Additionally the author is indebted to Lance Edward Miller, Neil Epstein, and Paolo Mantero for profitable discussions.

\section{The Multiplicity-Like Function $\rmuM{s}{I}{J}{M}$}\label{rmult section}

We begin by considering a limit which combines aspects of the limits defining the Hilbert-Samuel and Hilbert-Kunz multiplicities.  The idea is to take the colengths of a sum of ideals, one of which corresponds to the increasing Frobenius powers of an ideal $J$, and one of which corresponds to a subsequence of the powers of another ideal $I$.  This subsequence will be determined by a real number $s$.  We require that both of these ideals be primary to the maximal ideal of the ring they belong to so that at the extreme values of the parameter $s$ one of the two ideals will dominate the other. This guarantees that in the extremal cases we will get a limit related to either the Hilbert-Samuel multiplicity of $I$ or the Hilbert-Kunz multiplicity of $J$.

\begin{theorem} \label{theorem -- main existence}\standardassumptionsdimgreaterthan    The limit
\[\lim_{e\to\infty}\frac{\length{M/(I^{\lceil s p^e\rceil} +J^{[p^e]})M}}{p^{ed}}\]
exists.
\end{theorem}
To prove this we require a few results that will describe the generators of certain modules as $k$-vector spaces and establish some combinatorial facts which will allow us to effectively estimate the module lengths involved in the proof.

\begin{lemma} \label{lemma -- generator classification} Let $(R,\m,k)$ be a local ring containing its residue field, and let $M$ be an $R$-module of finite length. Let $\{x_1,\ldots, x_t\}$ be a set of generators for $\m$ and $\{m_1,\ldots, m_n\}$ a set of generators for $M$.   Then
\begin{enumerate}[(i)]
\item \label{generators of M} $M$ is generated as a $k$-vector space by elements of the form $x_1^{b_1}\cdots x_t^{b_t}m_j$, where $b_1,\ldots, b_t\in \N$ and $1\leq j\leq n$; and

\item \label{generators relative to I} If $I=(f_1,\ldots, f_m)$ is an $\m$-primary ideal of $R$ then $M$ is generated as a $k$-vector space by elements of the form 
$f_1^{a_1}\cdots f_m^{a_m}gm_j$, where $a_1,\ldots, a_m\in \N$, $1\leq j\leq n$, and $g$ is a generator of $R/I$ as a $k$-vector space. 
\end{enumerate}
\end{lemma}

\begin{proof}
\eqref{generators of M} By definition, $M$ is generated as a $k$-vector space by elements of the form $rm_j$ with $r\in R$ and $1\leq j \leq n$.  
For each such $r$, we have that $r=v+\sum_{i=1}^t r_ix_i$ for some $v\in k$ and $r_i\in R$, since $R=k\oplus \m$ as a $k$-vector space.  
For each $i$, we may write $r_i=v_i+\sum_{j=1}^n r_{ij}x_j$ with $v_i\in k$ and $r_{ij}\in R$, and so 
\[r=v+\sum_{i=1}^t v_ix_i+\sum_{1\leq i,j\leq t} r_{i,j}x_ix_j.\]
We may repeat this process until every term either has a coefficient of the $x_i$'s which is an element of $k$ or has a degree in the $x_i$'s large enough that the term annihilates $M$ and so may be removed.

\eqref{generators relative to I}
By part \eqref{generators of M}, $M$ is generated as a $k$-vector space by terms of the form $f_1^{a_1}\cdots f_m^{a_m}x_1^{b_1}\cdots x_t^{b_t}m_j$ with $a_i,b_i\in\N$.   Fix a set of $k$-vector space generators $\{g_i\}$ of $R/I$.  Suppose that we have an element $\alpha = f_1^{a_1}\cdots f_m^{a_m}x_1^{b_1}\cdots x_t^{b_t}m_j\in M$ with $x_1^{b_1}\cdots x_t^{b_t}\notin \{g_i\}$.  There exists $i$ such that $x_1^{b_1}\cdots x_t^{b_t}-g_i\in I$, and so there exist $r_1,\cdots,r_m\in R$ such that $x_1^{b_1}\cdots x_t^{b_t}-g_i = \sum_{\ell=1}^m r_\ell f_\ell.$
Therefore,
\[\alpha=f_1^{a_1}\cdots f_m^{a_m}g_im_j+\sum_{\ell=1}^m f_1^{a_1}\cdots f_\ell^{a_\ell+1}\cdots f_m^{a_m}r_\ell m_j.\]
We know by part (\ref{generators of M}) that $r_\ell m_j$ is a $k$-linear combination of terms of the form $x_1^{b_1'}\cdots x_t^{b_t'}m_{j'}$, and so we have that $\alpha$ is a $k$-linear combination of terms of the form $f_1^{a_1}\cdots f_m^{a_m}g_im_j$ and $f_1^{a_1'}\cdots f_m^{a_m'}x_1^{b_1'}\cdots x_t^{b_t'}m_{j'}$
with $\sum_\ell a_\ell'=1+\sum_\ell a_\ell$.
Continuing in this way, we may write $\alpha$ as a $k$-linear combination of terms either of the form $f_1^{a_1}\cdots f_m^{a_m}g_im_j$ for some $i$ or of the form $f_1^{a_1}\cdots f_m^{a_m}x_1^{b_1}\cdots x_t^{b_t}m_j$ with $\sum_i a_i$ arbitrarily large.  Since $I^n$ annihilates $M$ for some $n$, we may throw out all the terms of the second kind, which finishes the proof.
\end{proof}

Bounding the lengths of the ideals we are concerned with will involve some combinatorial calculations.  For convenience we introduce some notation.  
 For positive integers $d$ and $m$ and real number $r$, we set $S_d^m(r)$ to be the number of monomials in $d$ variables with degree less than $r$ and with degree in each variable less than $m$.

Certain properties of the numbers $S_d^m(r)$ are easy to see.  First, if $r\geq 0$, then $S_1^m(r)=\min\{m,\lceil r\rceil\}$.  Second, for $d>1$, we have that  
$S_d^m(r)=\sum_{i=0}^{m-1} S_{d-1}^m(r-i)$.  Indeed, if we denote one of the variables by $x$, then for $i=0,1,\ldots, m-1$, there are $S_{d-1}^m(r-i)$ monomials with degree exactly $i$ in $x$, degree less than $r$, and with degree in each variable less than $m$. 

We will occasionally use a combinatorial description of the numbers $S_m^d(r)$, which is established in the following lemma.  This result appeared in a more general form as \cite[Lemma 2.5]{RobinsonSwanson-ExplicitHKFunctionsDeterminantalRings},  though the method of proof was different.

\begin{lemma} \label{S_d^m(r) combinatorial} For positive integers $d$ and $m$ and real number $r$,
\[S_d^m(r)=\sum_{i=0}^d (-1)^i\BC{d}{i}\BC{\lceil r\rceil-im -1+d}{d}.\]
\end{lemma}

\begin{proof}
The number of monomials in $d$ variables, of degree less than $r$, where each of a given set of $i$ variables has degree at least $m$ is the number of monomials in $d$ variables of degree less than $r-im$, that is, $\BC{\lceil r\rceil-im -1+d}{d}$.  Thus the total number of monomials in $d$ variables of degree less than $r$ with degree in each variable less than $m$ is
\[\BC{\lceil r\rceil -1+d}{d}-\sum_{i=1}^d (-1)^{i-1} \BC{\lceil r\rceil -im-1-d}{d},\]
by the inclusion-exclusion principle.
\end{proof}

Our next lemma is a technical result on the behavior of the numbers $S_d^m(r)$ as $m$ and $r$ grow.

\begin{lemma} \label{lemma -- properties of S_d^m(r)} If $f,g:\N\to \R$ are functions such that $f(n)-g(n) \leq cn+o(n)$ for some $c\in\R$, $f(n)\geq g(n)$ for $n\gg0$, and $u$ is a positive integer, then
\[\limsup_{n\to\infty}\frac{S_d^{un}(f(n))-S_d^{un}(g(n))}{n^d}\leq u^{d-1}c.\]
\end{lemma}                                   

\begin{proof} We proceed by induction on $d$.  Suppose $d=1$, and let $n\in \N$ large enough that $f(n)\geq g(n)$.  If $un\leq g(n)$ we have that
\[S_1^{un}(f(n))-S_1^{un}(g(n))=0\]
and if $un>g(n)$ then
\[S_1^{un}(f(n))-S_1^{un}(g(n))\leq \lceil f(n)\rceil -\lceil g(n)\rceil\leq f(n)-g(n)+1.\]
Therefore
\[\limsup_{n\to\infty}\frac{S_1^{un}(f(n))-S_1^{un}(g(n))}{n}\leq \limsup_{n\to\infty}\frac{f(n)-g(n)+1}{n}\leq c.\]
Now suppose that $d>1$.  Then
\begin{align*}
S_d^{un}(f(n))-S_d^{un}(g(n))&=\sum_{i=0}^{un-1}\left(S_{d-1}^{un}(f(n)-i)-S_{d-1}^{un}(g(n)-i)\right)\leq un\left(S_{d-1}^{un}(f(n)-i_n)-S^{un}_{d-1}(g(n)-i_n)\right)
\end{align*}
where $i_n$ is the value of $i$ with $1\leq i\leq un-1$ that maximizes the expression $S_{d-1}^{un}(f(n)-i)-S_{d-1}^{un}(g(n)-i)$.  By induction,
\begin{align*}
\limsup_{n\to\infty}\frac{S_d^{un}(f(n))-S_d^{un}(g(n))}{n^d}
\leq & \limsup_{n\to\infty}\frac{un\left(S_{d-1}^{un}(f(n)-i_n)-S^{un}_{d-1}(g(n)-i_n)\right)}{n^d}\\
=&u\cdot \limsup_{n\to\infty}\frac{S_{d-1}^{un}(f(n)-i_n)-S^{un}_{d-1}(g(n)-i_n)}{n^{d-1}}\\
\leq &u\cdot u^{d-2}c=u^{d-1}c.\qedhere
\end{align*}
\end{proof}
 
\begin{proof}[Proof of Theorem \ref{theorem -- main existence}]
If $d=0$, then for large enough $e$, $\ii{s}{p^e}=0$ and so the limit is simply $\length{R}$.  Suppose that $d\geq 1$. 
If $k$ is not infinite, we may replace $R$ by $S=R[X]_{\m R[X]}$.  For any $R$-module $N$, we have $\lengthR{R}{N}=\lengthR{S}{N\otimes_R S}$, and so we may assume without loss of generality that the ring $R$ has infinite residue field.
Let $K$ be a reduction of $I$ generated by $d$ elements $f_1,\ldots, f_d\in R$, and let $w$ be the reduction number of $I$ with respect to $K$.
Let $x_1,\ldots, x_t\in R$ be a set of generators for the maximal ideal $\m$.
Let $m_1,\ldots,m_n\in M$ be a set of generators of $M$.
Let $v\in \N$ such that $K^v\subseteq J$.
Let $q,q'$ be varying powers of $p$.

If $q'>\frac{w+d}{s}$, then for sufficiently large $q$ we have that
\[\left(K^{\lceil sq'\rceil}+J^{[q']}\right)^{[q]}\subseteq \left(I^{\lceil sq'\rceil }+J^{[q']}\right)^{[q]}\subseteq I^{\lceil sq'q\rceil}+J^{[q'q]}
\subseteq K^{\lceil sq'q\rceil - w}+J^{[q'q]}\subseteq \left(K^{\lceil sq'\rceil-d-1}+J^{[q']}\right)^{[q]}.\]
Therefore,
\begin{align*}
\length{\frac{M}{\left(K^{\lceil sq'\rceil-d-1}+J^{[q']}\right)^{[q]}M}} &\leq \length{\frac{M}{\left(I^{\lceil sq'q\rceil}+J^{[q'q]}\right)M}} \leq \length{\frac{M}{\left(K^{\lceil sq'\rceil}+J^{[q']}\right)^{[q]}M}}.\end{align*}
If we divide the first and last terms of this inequality by $q^d$, then the limit as $q\to\infty$ exists by \cite[Theorem 1.8]{Monsky-HilbertKunzFunction}.  Hence
\begin{align*}
&\limsup_{q\to\infty} \frac{1}{q^d} \length{\frac{M}{\left(I^{\lceil sq'q\rceil}+J^{[q'q]}\right)M}}- \liminf_{q\to\infty}\frac{1}{q^d} \length{\frac{M}{\left(I^{\lceil sq'q\rceil}+J^{[q'q]}\right)M}}\\
\leq &\lim_{q\to\infty}\frac{1}{q^d}\left(\length{\frac{M}{\left(K^{\lceil sq'\rceil}+J^{[q']}\right)^{[q]}M}}-\length{\frac{M}{\left(K^{\lceil sq'\rceil-d-1}+J^{[q']}\right)^{[q]}M}}\right)\\
=&\lim_{q\to\infty}\frac{1}{q^d}\length{\frac{\left(K^{\lceil sq'\rceil-d-1}+J^{[q']}\right)^{[q]}M}{\left(K^{\lceil sq'\rceil}+J^{[q']}\right)^{[q]}M}}.
\end{align*}
Let 
\[Q=\frac{\left(K^{\lceil sq'\rceil-d-1}+J^{[q']}\right)^{[q]}M}{\left(K^{\lceil sq'\rceil}+J^{[q']}\right)^{[q]}M}\cong \frac{\left(K^{[q]}\right)^{\lceil sq'\rceil-d-1}M}{\left(\left(K^{[q]}\right)^{\lceil sq'\rceil}+J^{[q'q]}\right)M\cap \left(K^{[q]}\right)^{\lceil sq'\rceil-d-1} M}.\]
As an $R$-module, $Q$ is generated by elements of the form $f_1^{y_1q}\cdots f_d^{y_dq}m_\alpha$,
where $\sum_i y_i = \lceil sq'\rceil -d-1$ and $1\leq \alpha\leq n$.  Therefore, by Lemma \ref{lemma -- generator classification}, $Q$ can be generated as a $k$-vector space by elements of the form 
$f_1^{y_1q+z_1}\cdots f_d^{y_dq+z_d}gm_\alpha$
where $b_i,y_i,z_i\in \N$, $\sum_i y_i = \lceil sq'\rceil -d-1$, and $g$ is a $k$-vector space generator of $R/K$.  Letting $c_i=y_i+\lfloor z_i/q\rfloor$ and $a_i=z_i-q\lfloor z_i/q\rfloor$, we have that $c_iq+a_i=y_iq+z_i$ and $a_i<q$, and so $Q$ can be generated as a $k$-vector space by elements of the form
$f_1^{c_1q+a_1}\cdots f_d^{c_dq+a_d}gm_\alpha$
where $a_i,b_i,c_i\in \N$, $a_i<q$, $\sum_i c_i\geq \lceil sq'\rceil -d-1$, $g$ is a $k$-vector space generator of $R/K$, and $1\leq \alpha\leq n$.  However, if $\sum_i c_i\geq sq'$ or $c_i\geq vq'$ for some $i$, then the product above vanishes in $Q$.  Therefore
\[\length{Q}\leq q^d\cdot \left(S_d^{vq'}(sq')-S_d^{vq'}(sq'-d-1)\right)\cdot \length{R/K}\cdot n.\]
From this we have that
\begin{align*}
&\limsup_{q\to\infty} \frac{1}{q^d} \length{\frac{M}{\left(I^{\lceil sq\rceil}+J^{[q]}\right)M}}- \liminf_{q\to\infty}\frac{1}{q^d} \length{\frac{M}{\left(I^{\lceil sq\rceil}+J^{[q]}\right)M}}\\
=&\limsup_{q\to\infty} \frac{1}{(q'q)^d} \length{\frac{M}{\left(I^{\lceil sq'q\rceil}+J^{[q'q]}\right)M}}- \liminf_{q\to\infty}\frac{1}{(q'q)^d} \length{\frac{M}{\left(I^{\lceil sq'q\rceil}+J^{[q'q]}\right)M}}\\
\leq & \lim_{q\to\infty}\frac{q^d\cdot \left(S_d^{vq'}(sq')-S_d^{vq'}(sq'-d-1)\right)\cdot \length{R/K}\cdot n}{(q'q)^d}\\
=&\frac{\left(S_d^{vq'}(sq')-S_d^{vq'}(sq'-d-1)\right)\cdot \length{R/K}\cdot n}{(q')^d}.\end{align*}
Since this holds for all $q'\gg 0$, and by Lemma \ref{lemma -- properties of S_d^m(r)}, 
\begin{align*}
&\limsup_{q\to\infty} \frac{1}{q^d} \length{\frac{M}{\left(I^{\lceil sq\rceil}+J^{[q]}\right)M}}- \liminf_{q\to\infty}\frac{1}{q^d} \length{\frac{M}{\left(I^{\lceil sq\rceil}+J^{[q]}\right)M}}\\
\leq&\limsup_{q'\to\infty}\frac{\left(S_d^{vq'}(sq')-S_d^{vq'}(sq'-d-1)\right)\cdot \length{R/K}\cdot n}{(q')^d}
\leq 0.
\end{align*}
Thus the limit exists and the theorem is proved.
\end{proof}

\begin{definition} \standardassumptionsnosdim  For $s>0$, we set
\[\rmuM{s}{I}{J}{M} = \lim_{e\to\infty}\frac{\length{M/(I^{\lceil sp^e\rceil} +J^{[p^e]})M}}{p^{ed}}.\]
We will often write $\rmu{s}{I}{J}$ for $\rmuM{s}{I}{J}{R}$, $\rmuuM{s}{I}{M}$ for $\rmuM{s}{I}{I}{M}$, $\rmuu{s}{I}$ for $\rmuuM{s}{I}{R}$, and $\rmuu{s}{M}$ for $\rmuuM{s}{\m}{M}$.  If we wish to emphasize the ring $R$, we will write $\rmuRM{R}{s}{I}{J}{M}$ or a similarly decorated variant.\end{definition}

We next establish some properties of $\rmuM{s}{I}{J}{M}$.  We will use the next result repeatedly throughout the paper, often without explicit reference.

\begin{proposition} \label{rmult basics} \standardassumptionsnosdim The following statements hold:
\begin{enumerate}[(i)]
\item \label{rmult HS compare} \label{rmult HK compare} $\rmuM{s}{I}{J}{M}\leq \min\{\frac{s^d}{d!}e(I;M),e_{HK}(J;M)\}$.
\item \label{rmult low dim} If $\dim M<d$ then $\rmuM{s}{I}{J}{M}=0$.
\item \label{rmult s increasing} If $s'\geq s$ then $\rmuM{s'}{I}{J}{M}\geq \rmuM{s}{I}{J}{M}$.
\item \label{rmult ideals decreasing} If $I'$ and $J'$ are ideals of $R$ such that $I\subseteq I'$ and $J\subseteq J'$, then $\rmuM{s}{I'}{J'}{M}\leq \rmuM{s}{I}{J}{M}$. 
\item \label{rmult integral closure} If $I'$ is an ideal of $R$ with the same integral closure as $I$, then $\rmuM{s}{I'}{J}{M}=\rmuM{s}{I}{J}{M}$.
\item \label{rmult tight closure} If $J'$ is an ideal of $R$ with the same tight closure as $J$, then $\rmuM{s}{I}{J'}{M}=\rmuM{s}{I}{J}{M}$.
\end{enumerate}
\end{proposition}

\begin{proof}

(\ref{rmult HS compare}) For all $e\in \N$ we have that $\ii{s}{p^e}\supseteq I^{\lceil sp^e\rceil}$, hence 
\[\lim_{e\to\infty}\frac{\length{M/(\ii{s}{p^e})M}}{p^{ed}}\leq\lim_{e\to\infty}\frac{\length{M/I^{\lceil sp^e\rceil }M}}{\lceil sp^e\rceil^{d}}\cdot\frac{\lceil sp^e\rceil^{d}}{p^{ed}}=\frac{s^d}{d!}e(I;M).\]
Furthermore, for all $e\in \N$ we have that $\ii{s}{p^e}\supseteq J^{[p^e]}$, hence 
\[\lim_{e\to \infty} \frac{\length{M/(\ii{s}{p^e})M}}{p^{ed}}\leq \lim_{e\to\infty}\frac{\length{M/J^{[p^e]}M}}{p^{ed}}=e_{HK}(J;M).\] 

(\ref{rmult low dim}) By \cite[Lemma 1.2]{Monsky-HilbertKunzFunction}, $e_{HK}(J;M)=0$ for any $M$ with $\dim M<d$, and so part (\ref{rmult HK compare}) gives us the result.

(\ref{rmult s increasing}) For all $e\in \N$ we have that $\ii{s}{p^e}\supseteq \ii{s'}{p^e}$, hence \[\length{M/(\ii{s}{p^e})M}\leq \length{M/(\ii{s'}{p^e})M}.\]

(\ref{rmult ideals decreasing}) For all $e\in \N$ we have that $\iii{s}{p^e}{I'}{J'}\supseteq \ii{s}{p^e}$, hence \[\length{M/(\iii{s}{p^e}{I'}{J'})M}\leq \length{M/(\ii{s}{p^e})M}.\] 

(\ref{rmult integral closure}) It suffices to prove the case where $I'=\overline{I}$, the integral closure of $I$.  If $s>0$, then we have that, by part (\ref{rmult ideals decreasing}) and \cite[Proposition 11.2.1]{HunekeSwanson-IntegralClosureIdealsRingsModules},
\begin{align*}
0&\leq \rmuM{s}{I}{J}{M}-\rmuM{s}{\overline{I}}{J}{M}
=\lim_{e\to\infty}\frac{1}{p^{ed}}\length{\frac{\iii{s}{p^e}{\overline{I}}{J}}{\ii{s}{p^e}}}\leq \lim_{e\to\infty}\frac{1}{p^{ed}}\length{\frac{(\overline{I})^{\lceil sp^e\rceil}}{I^{\lceil sp^e\rceil}}}
=\frac{s^d}{d!}\left(e(I)-e(\overline I)\right)=0.
\end{align*}

(\ref{rmult tight closure}) It suffices to prove the case where $J=J^\ast$, the tight closure of $J$.  We have that,  by part (\ref{rmult ideals decreasing}) and \cite[Theorem 8.17]{HochsterHuneke-TightClosureInvariantTheory},
\begin{align*}
0&\leq \rmuM{s}{I}{J}{M}-\rmuM{s}{I}{J^\ast}{M}\\&
=\lim_{e\to\infty}\frac{1}{p^{ed}}\length{\frac{\iii{s}{p^e}{I}{(J^\ast)}}{\ii{s}{p^e}}}\leq \lim_{e\to\infty}\frac{1}{p^{ed}}\length{\frac{(J^\ast)^{[p^e]}}{J^{[p^e]}}}
=e_{HK}(J)-e_{HK}(J^\ast)=0.\qedhere
\end{align*}
\end{proof}

\begin{theorem} \label{theorem -- rmult lipschitz} \standardassumptionsnos The function $\rmuM{s}{I}{J}{M}$ is Lipschitz continuous.
\end{theorem}

\begin{proof} Let $\delta>0$.  The function $\rmuM{s}{I}{J}{M}$ is increasing by Proposition \ref{rmult basics}(\ref{rmult s increasing}), so we need only bound $\rmuM{s+\delta}{I}{J}{M}-\rmuM{s}{I}{J}{M}$ above in terms of $\delta$.  

Let $d=\dim R$. If $d=0$ then $\rmuM{s+\delta}{I}{J}{M}=\rmuM{s}{I}{J}{M}=\length{M}$, so 0 is a Lipschitz constant for $\rmuM{s}{I}{J}{M}$.  Suppose $d\geq 1$.
 We may assume that $R/\m$ is infinite, and so we may assume that $I$ is generated by $d$ elements by replacing it with a minimal reduction by Proposition \ref{rmult basics}(\ref{rmult integral closure}).  Let $I=(f_1,\ldots, f_d)$, let $\m=(x_1,\ldots, x_t)$, let $v\in \N$ such that $I^v\subseteq J$, and let $m_1,\ldots, m_n$ be a set of generators for $M$.  Then
\begin{align*}
\rmuM{s+\delta}{I}{J}{M}-\rmuM{s}{I}{J}{M}
=&\lim_{e\to\infty}\frac{1}{p^{ed}}\left(\length{M/(\ii{(s+\delta)}{p^e})M}-\length{M/(\ii{s}{p^e})M}\right)\\
=&\lim_{e\to\infty}\frac{1}{p^{ed}}\length{\frac{(\ii{s}{p^e})M}{(\ii{(s+\delta)}{p^e})M}}\\
=&\lim_{e\to\infty}\frac{1}{p^{ed}}\length{\frac{I^{\left\lceil sp^e\right\rceil} M}{(\ii{(s+\delta)}{p^e})M\cap I^{\left\lceil sp^e\right \rceil}M}}.
\end{align*}
The quotient module in the last line is generated as a $k$-vector space by elements of the form $f_1^{a_1}\cdots f_d^{a_d}gm_\alpha$, where $\sum_i a_i\geq sp^e$, $g$ is a $k$-vector space generator of $R/I$, and  $1\leq \alpha\leq n$.  However, if $\sum_i a_i\geq (s+\delta)p^e$ or $a_i\geq vp^e$ for some $i$, then the corresponding product vanishes.  Therefore,
\[\length{\frac{I^{\left\lceil sp^e\right\rceil} M}{(\ii{(s+\delta)}{p^e})M\cap I^{\left\lceil sp^e\right \rceil}M}}
\leq \left(S_d^{vp^e}((s+\delta)p^e)-S_d^{vp^e}(sp^e)\right)\cdot \length{R/I}\cdot n,\]
and so, by Lemma \ref{lemma -- properties of S_d^m(r)},
\begin{align*}
\rmuM{s+\delta}{I}{J}{M}-\rmuM{s}{I}{J}{M}&\leq \limsup_{e\to\infty}\frac{ (S_d^{vp^e}((s+\delta)p^e)-S_d^{vp^e}(sp^e))\cdot\length{R/I}\cdot n}{p^{ed}}\leq \delta \cdot v^{d-1}\cdot \length{R/I} \cdot n.\end{align*} Hence $v^{d-1}\cdot \length{R/I} \cdot n$ is a Lipschitz constant for $\rmuM{s}{I}{J}{M}$.
\end{proof}

Our most important application of Theorem \ref{theorem -- rmult lipschitz} is the next result, which proves that $\rmuM{s}{I}{J}{M}$ is additive on short exact sequences.  A direct consequence of this will be the Associativity Formula for $s$-multiplicity.

\begin{theorem}\label{rmult additivity} \standardassumptionsnosnom If $0\to M'\to M \to M''\to 0$ is a short exact sequence of finitely generated $R$-modules, then $\rmuM{s}{I}{J}{M}=\rmuM{s}{I}{J}{M'}+\rmuM{s}{I}{J}{M''}$.
\end{theorem}
 
 \begin{proof} Let $d=\dim R$, let $m$ be the minimal number of generators of $I$, and fix $e\in \N$.  For any $e'\in\N$, we have that
 \[\ii{(s+m/p^e)}{p^{e+e'}}\subseteq \left(\ii{s}{p^e}\right)^{[p^{e'}]}\subseteq \ii{s}{p^{e+e'}}.\]
 By \cite[Theorem 1.6]{Monsky-HilbertKunzFunction}, we have that
 \begin{align*}
 &\length{\frac{M'}{(\ii{s}{p^{e+e'}})M'}}+\length{\frac{M''}{(\ii{s}{p^{e+e'}})M''}}\\
 \leq &\;\length{\frac{M'}{ \left(\ii{s}{p^e}\right)^{[p^{e'}]}M'}}+ \length{\frac{M''}{\left(\ii{s}{p^e}\right)^{[p^{e'}]}M''}}\\
 =& \;\length{\frac{M}{ \left(\ii{s}{p^e}\right)^{[p^{e'}]}M}}+O(p^{e'(d-1)})\\
 \leq &\; \length{\frac{M}{(\ii{(s+m/p^e)}{p^{e+e'}})M}}+O(p^{e'(d-1)}).
 \end{align*}
 Dividing by $p^{(e+e')d}$ and taking the limit as $e'\to \infty$, we obtain that
 \[\rmuM{s}{I}{J}{M'}+\rmuM{s}{I}{J}{M''}\leq \rmuM{s+m/p^e}{I}{J}{M}.\]
This holds for all $e$, and so $\rmuM{s}{I}{J}{M'}+\rmuM{s}{I}{J}{M''}\leq \rmuM{s}{I}{J}{M}$ since by Theorem \ref{theorem -- rmult lipschitz}, $\rmuM{s}{I}{J}{M}$ is continuous in $s$.
 
 For the other inequality, note that for any $e\in \N$, the sequence
 \[\frac{M'}{(\ii{s}{p^e})M'} \to \frac{M}{(\ii{s}{p^e})M}\to \frac{M''}{(\ii{s}{p^e})M''} \to 0\]
 is exact, whence
 \begin{align*}
 &\length{\frac{M'}{(\ii{s}{p^e})M'}}+\length{ \frac{M''}{(\ii{s}{p^e})M''}} \geq \length{ \frac{M}{(\ii{s}{p^e})M}}.\end{align*}
 Therefore $\rmuM{s}{I}{J}{M'}+\rmuM{s}{I}{J}{M''}\geq \rmuM{s}{I}{J}{M}$.
 \end{proof}
 
 The additivity of $\rmuM{s}{I}{J}{M}$ on short exact sequences is exactly what we need to prove the Associativity Formula for $s$-multiplicity.  This proof follows the proof in \cite[Theorem 23.5]{Nagata-LocalRings} for the Associativity Formula for Hilbert-Samuel multiplicity.
 
\begin{theorem}[The Associativity Formula] \label{rmult Associativity Formula} \standardassumptionsnos We have that
\[\rmuRM{R}{s}{I}{J}{M}=\sum_{\p\in \Assh R}\rmuR{R/\p}{s}{I(R/\p)}{J(R/\p)}\lengthR{R_\p}{M_\p}\]
where $\Assh R=\set{\p\in \Spec R}{\dim R/\p=\dim R}$.
\end{theorem}

\begin{proof} We proceed by induction on $\sigma(M)=\sum_{\p\in\Assh R}\lengthR{R_\p}{M_\p}$.  If $\sigma(M)=0$, then $\dim M<\dim R$ and so $\rmuRM{R}{s}{I}{J}{M}=0$.

Now suppose that $\sigma(M)\geq 1$ and fix $\q\in \Assh R$ such that $\lengthR{R_\q}{M_\q}\geq 1$.  Then $\q=(0:_R x)$ for some $x\in M$ and so we have an exact sequence
\begin{equation*}  
0\to R/\q\to M\to M/Rx\to 0.\end{equation*}
We have that $\sigma(M/Rx)=\sigma(M)-1$ and so by induction, 
\begin{align*}\rmuRM{R}{s}{I}{J}{M/Rx}=&\sum_{\p\in\Assh R}\rmuR{R/\p}{s}{I(R/\p)}{J(R/\p)}\lengthR{R_\p}{(M/Rx)_\p}\\
=&
\sum_{\p\in\Assh R}\rmuR{R/\p}{s}{I(R/\p)}{J(R/\p)}\lengthR{R_\p}{M_\p}-\rmuR{R/\q}{s}{I(R/\q)}{J(R/\q)}.\end{align*}
Therefore, it suffices to show that $\rmuRM{R}{s}{I}{J}{R/\q}=\rmuR{R/\q}{s}{I(R/\q)}{J(R/\q)}$ since then by 
Theorem \ref{rmult additivity} we will have the desired formula.  This, however, is an easy computation:
\begin{align*}
\rmuRM{R}{s}{I}{J}{R/\q}
=& \lim_{e\to\infty} \frac{1}{p^{ed}}\lengthR{R}{\frac{R/\q}{(\ii{s}{p^e})R/\q}}\\
=&\lim_{e\to\infty} \frac{1}{p^{ed}}\lengthR{R/\q}{\frac{R/\q}{\iii{s}{p^e}{(I(R/\q))}{(J(R/\q))}}}
=\rmuR{R/\q}{s}{I(R/\q)}{J(R/\q)}. \qedhere
\end{align*}
\end{proof}

\section{$s$-Multiplicity}\label{smult section}

The behavior of $\rmuM{s}{I}{J}{M}$ is related to two thresholds concerning the interactions between powers and Frobenius powers of ideals.

\begin{definition}(\cite{MustataTakagiWatanabe-FthresholdsBernsteinSatoPolynomials}, \cite{dSNBP-ExistenceofFthresholds}) Let $R$ be a ring of characteristic $p>0$, and let $I,J$ be ideals of $R$.  For $e\in \N$, let 
\[\nu_J^I(p^e)=\sup\set{n\in\N}{I^n\not\subseteq J^{[p^e]}}\qquad
\text{and}\qquad
\mu_J^I(p^e)=\inf\set{n\in\N}{J^{[p^e]}\not\subseteq I^n}.\]
The \emph{$F$-threshold of $I$ with respect to $J$} is $\displaystyle{\fthresh{I}{J}=\lim_{e\to\infty}\frac{\nu_J^I(e)}{p^e}}$.
Similarly, we set
$\displaystyle{\sthresh{I}{J}=\lim_{e\to\infty}\frac{\mu_J^I(e)}{p^e}}$.
\end{definition}

\begin{lemma} Let $R$ be a ring of characteristic $p>0$, and let $I,J$ be ideals of $R$.  The limits defining $\fthresh{I}{J}$ and $\sthresh{I}{J}$ are defined.  Furthermore, if $I\not\subseteq \sqrt{J}$ then $\fthresh{I}{J}=\infty$, if $J=R$  then $\fthresh{I}{J}=-\infty$, and if $I\subseteq \sqrt{J}\neq R$ then $0\leq \fthresh{I}{J}<\infty$.  Similarly, if $J\not\subseteq \sqrt{I}$, then $\sthresh{I}{J}=0$, if $I=R$ then $\sthresh{I}{J}=\infty$, and if $J\subseteq \sqrt{I}\neq R$ then $\sthresh{I}{J}>0$.  If $I\not\subseteq \sqrt{0}$, $I\subseteq\sqrt{J}$, $J\subseteq \sqrt{I}$, and $I$ is contained in the Jacobson radical of $R$, then $\sthresh{I}{J}\leq \fthresh{I}{J}$.
\end{lemma}

\begin{proof} If $I\not\subseteq \sqrt{J}$, then $\nu_J^I(p^e)=\infty$ for all $e$ and so $\fthresh{I}{J}=\infty$. If $J=R$ then $\nu_J^I(p^e)=-\infty$ for all $e$ and so $\fthresh{I}{J}=-\infty$.  Suppose $1\subseteq \sqrt{J}\neq R$, so that for all $e$, $0\leq \nu_J^I(p^e)$ and so $\fthresh{I}{J}\geq 0$. That $\fthresh{I}{J}$ exists in the case $I\subseteq \sqrt{J}$ is \cite[Theorem 3.4]{dSNBP-ExistenceofFthresholds}, the proof of which also shows that $\fthresh{I}{J}<\infty$ in this case.

If $J\not\subseteq \sqrt{I}$, then $\mu_J^I(p^e)=1$ for all $e$ and so $\sthresh{I}{J}=0$.  If $I=R$, then $\mu_J^I(p^e)=\infty$ for all $e$, and so $\sthresh{I}{J}=\infty$.  Suppose that $J\subseteq \sqrt{I}\neq R$.  The proof of the existence of $\sthresh{I}{J}$ in this case is nearly identical to that of the existence of $\fthresh{I}{J}$.   Let $e,e'\in \N$.  We have that $J^{[p^{e+e'}]}=\left(J^{[p^{e'}]}\right)^{[p^{e}]}\subseteq \left(I^{\mu_J^I(p^{e'})-1}\right)^{[p^{e}]}\subseteq I^{p^e\mu_J^I(p^{e'})-p^e}$, and so $\mu_J^I(p^{e+e'})>p^e\mu_J^I(p^{e'})-p^e$.  Therefore,
\[\liminf_{e\to\infty}\frac{\mu_J^I(p^e)}{p^e}=\liminf_{e\to\infty}\frac{\mu_J^I(p^{e+e'})}{p^{e+e'}}\geq \lim_{e\to\infty}\frac{\mu_J^I(p^{e'})-1}{p^{e'}}=\frac{\mu_J^I(p^{e'})-1}{p^{e'}}\]
Hence $\displaystyle{\liminf_{e\to\infty}\frac{\mu_J^I(p^e)}{p^e}\geq \limsup_{e'\to\infty}\frac{\mu_J^I(p^{e'})-1}{p^{e'}}}=\limsup_{e'\to\infty}\frac{\mu_J^I(p^{e'})}{p^{e'}}$ and so the limit defining $\sthresh{I}{J}$ exists.  Since $J\subseteq \sqrt{I}$, there exists $e\in \N$ such that $J^{[p^e]}\subseteq I$, and so $\mu_J^I(p^e)\geq 2$.  Hence, $\sthresh{I}{J}\geq \frac{\mu_J^I(p^{e})-1}{p^{e}}>0$.

For the last statement, suppose $I\not\subseteq \sqrt{0}$, $ I\subseteq \sqrt{J}$, $J\subseteq \sqrt{I}$, and $I$ is in the Jacobson radical of $R$.  For $e\in \N$ , we have that $I^{\nu_J^I(p^e)+1}\subseteq J^{[p^e]}\subseteq I^{\mu_J^I(p^e)-1}$.  By Nakayama's Lemma, we have that $\nu_J^I(p^e)+1\geq \mu_J^I(p^e)-1$.  Therefore, 
\[\fthresh{I}{J}=\lim_{e\to\infty}\frac{\nu_J^I(p^e)}{p^e}\geq \lim_{e\to\infty}\frac{\mu_J^I(p^e)-2}{p^e}=\sthresh{I}{J}.\qedhere\]
\end{proof}

\begin{lemma} \label{lemma -- thresholds vs rmult} \standardassumptionsnosdim
\begin{enumerate}
\item If $s\leq\sthresh{I}{J}$ then $\rmuM{s}{I}{J}{M}=\frac{s^d}{d!}e(I;M)$.
\item If $s\geq\fthresh{I}{J}$ then $\rmuM{s}{I}{J}{M}=e_{HK}(J;M)$.
\end{enumerate}
\end{lemma}

\begin{proof} If $s<\sthresh{I}{J}$, then for infinitely many $e\in\N$, $\mu_J^I(p^e)>\lceil sp^e \rceil$, and so $J^{[p^e]}\subseteq I^{\lceil sp^e\rceil}$.  Therefore
\begin{align*}
\rmuM{s}{I}{J}{M}
=&  \lim_{e\to\infty}\frac{\length{M/(I^{\lceil sp^e\rceil} +J^{[p^e]})M}}{p^{ed}}\\
=&\lim_{e\to\infty}\frac{\length{M/I^{\lceil sp^e\rceil}M}}{p^{ed}}
=\lim_{e\to\infty}\frac{\length{M/I^{\lceil sp^e\rceil}M}}{(\lceil sp^e\rceil)^d}\cdot\frac{(\lceil sp^e\rceil)^d}{p^{ed}}
=\frac{e(I;M)s^d}{d!}.
\end{align*}
If $s>\fthresh{I}{J}$, then for infinitely many $e\in\N$, $\nu_J^I(p^e)<\lceil sp^e\rceil$, and so $I^{\lceil sp^e\rceil}\subseteq J^{[p^e]}$.   Therefore
\[\rmuM{s}{I}{J}{M}
=  \lim_{e\to\infty}\frac{\length{M/(I^{\lceil sp^e\rceil} +J^{[p^e]})M}}{p^{ed}}
=\lim_{e\to\infty}\frac{\length{M/J^{[p^e]}M}}{p^{ed}}=e_{HK}(J;M).\]
The continuity of $\rmuM{s}{I}{J}{M}$ gives the cases $s=\sthresh{I}{J}$ and $s=\fthresh{I}{J}$.
\end{proof}

When $s$ is large, then $\rmuM{s}{I}{J}{M}$ precisely equals $e_{HK}(J)$, while when $s$ is small it equals a well-understood multiple of $e(I)$ depending only on $s$ and the dimension of the ring. Hence, in order to properly interpolate between the two functions we need a normalizing factor that will take this difference in behavior into account.
To determine a good candidate for this factor, we look at one of the most notable properties of $e(-)$ and $e_{HK}(-)$, namely, if $(R,\m)$ is a regular local ring of positive characteristic, then $e(\m)=e_{HK}(\m)=1$.   To that end, we calculate $\rmuu{s}{R}$ for power series rings over a field.

\begin{proposition} \label{H_s(d) combinatorial} If $k$ is a field of characteristic $p>0$ and $R=k[[x_1,\ldots,x_d]]$, then
\[\displaystyle{\rmuu{s}{R}=\sum_{i=0}^{\lfloor s\rfloor}\frac{(-1)^i}{d!}\BC{d}{i}(s-i)^d}.\]
\end{proposition}

\begin{proof} Let $\m=(x_1,\ldots,x_d)$.  If $d=0$, then $\m=0$, and so $\rmuu{s}{R}=1=\sum_{i=0}^{\lfloor s\rfloor}(-1)^i\BC{0}{i}(s-i)^0$.  If $d\geq 1$, then by Lemma \ref{S_d^m(r) combinatorial} we have that
\begin{align*}
\rmuu{s}{R} = \lim_{e\to \infty}\frac{S_d^m(sp^e)}{p^{ed}}
&=\sum_{i=0}^d (-1)^i\BC{d}{i}\lim_{e\to \infty}\frac{1}{p^{ed}}\BC{\lceil sp^e\rceil -ip^e-1+d}{d}=\sum_{i=0}^{\lfloor s\rfloor}\frac{(-1)^i}{d!}\BC{d}{i}(s-i)^d.\qedhere
\end{align*}
\end{proof}

Proposition \ref{H_s(d) combinatorial} gives us our normalizing factor, and so we are ready to define the $s$-multiplicity.

\begin{definition} \label{definition -- s-multiplicity} \standardassumptionsgreaterthan  Then the \emph{$s$-multiplicity} of $M$ with respect to the pair $(I,J)$ is defined to be
\[\smuM{s}{I}{J}{M}=\frac{\rmuM{s}{I}{J}{M}}{\nfact{d}{s}},\]
where 
$\displaystyle{\rmuu{s}{R}=\sum_{i=0}^{\lfloor s\rfloor}\frac{(-1)^i}{d!}\BC{d}{i}(s-i)^d}$.
We may write $\smu{s}{I}{J}$ for $\smuM{s}{I}{J}{R}$, $\smuuM{s}{I}{M}$ for $\smuM{s}{I}{I}{M}$, $\smuu{s}{I}$ for $\smuuM{s}{I}{R}$, and $\smuu{s}{M}$ for $\smuuM{s}{\m}{M}$.  If we wish to emphasize the ring $R$, we will write $\smuRM{R}{s}{I}{J}{M}$ or a similarly decorated variant.\end{definition}

In order to describe the interpolating properties of the $s$-multiplicity, we need some additional facts about the functions $\nfact{d}{s}$.   First we describe the functions explicitly for $d$ up to 3:

\begin{example} 
\begin{align*}
\nfact{0}{s}=&\;1\\
\nfact{1}{s}=&\begin{cases} s & \text{ if } 0<s< 1\\
1 & \text { if } s\geq 1\end{cases}\\
\end{align*}
\begin{align*}
\nfact{2}{s}=&\begin{cases}  \frac{1}{2}s^2 & \text{ if } 0<s< 1\\ \frac{1}{2}s^2-(s-1)^2 & \text{ if } 1\leq s < 2 \\  1 & \text{ if } s\geq 2\end{cases}\\
\nfact{3}{s}=&\begin{cases}  \frac{1}{6}s^3 & \text{ if } 0<s<1\\ \frac{1}{6}s^3-\frac{1}{2}(s-1)^3 & \text{ if } 1\leq s<  2 \\  \frac{1}{6}s^3-\frac{1}{2}(s-1)^3+\frac{1}{2}(s-2)^3 & \text{ if } 2\leq s< 3 \\ 1 & \text{ if } s\geq 3\end{cases}
\end{align*}
\end{example}

Certain properties of $\nfact{d}{s}$ are suggested by the above examples, and are confirmed in the next lemma.

\begin{lemma} \label{lemma -- properties of E_d(s)}  The functions $\nfact{d}{s}$ have the following properties.
\begin{enumerate}[(i)]
\item \label{nfact recursion} If $d\geq 1$, then $\nfact{d}{s}=\int_{s-1}^s\nfact{d-1}{t}\mathrm{d}t$.
\item \label{nfact increasing} $\nfact{d}{s}$ is nondecreasing.
\item \label{item continuous} $\nfact{d}{s}$ is a Lipschitz continuous function of $s$.
\item \label{item right tail} If $s\geq d$, then $\nfact{d}{s}=1$.
\item \label{item first interval behavior} If $0<s\leq 1$, then $\nfact{d}{s}=s^d/d!$.
\end{enumerate}
\end{lemma}

\begin{proof}

(\ref{nfact recursion}) This is clear for $d=1$, so suppose that $d\geq 2$.   Let $q$ and $q'$ be varying powers of $p$.  We have that 
\begin{align*}
\nfact{d}{s}
 =\lim_{q\to\infty} \frac{S_d^{qq'}(sqq')}{(qq')^d}
&=\lim_{q\to\infty} \frac{\sum_{i=0}^{qq'-1}S_{d-1}^{qq'}(sqq'-i)}{(qq')^d}\\
 &\leq \lim_{q\to\infty} \frac{q\sum_{i=0}^{q'-1}S_{d-1}^{qq'}(sqq'-qi)}{(qq')^d}\\
&=\frac{1}{q'}\sum_{i=0}^{q'-1}\lim_{q\to\infty}\frac{S_{d-1}^{qq'}\left((s-i/q')qq'\right)}{(qq')^{d-1}}
=\frac{1}{q'}\sum_{i=0}^{q'-1}\nfact{d-1}{s-i/q'}
\end{align*}
Since the above holds for all $q'$, we have that
\[\nfact{d}{s}\leq \lim_{q'\to\infty}\frac{1}{q'}\sum_{i=0}^{q'-1}\nfact{d-1}{s-i/q'}=\int^s_{s-1}\nfact{d}{t}\mathrm{d}t.\]
A similar argument, only using the inequality 
\[\sum_{i=0}^{qq'-1}S_{d-1}^{qq'}(sqq'-i)\geq q\sum_{i=1}^{q'}S_{d-1}^{qq'}(sqq'-qi)\]
in the second line, shows that $\nfact{d}{s}\geq \int^s_{s-1}\nfact{d}{t}\mathrm{d}t$.

(\ref{nfact increasing}) This is by inspection for $d=0$.  For $d\geq 1$, let $\delta>0$, so by induction
\[\nfact{d}{s+\delta}-\nfact{d}{s}=\int_{s-1}^s \nfact{d-1}{t+\delta}-\nfact{d-1}{t}\mathrm{d}t\geq 0.\]

(\ref{item continuous}) We claim that the functions $\nfact{d}{s}$ have Lipshitz constants at most 1.  This is trivial for $d=0$, so suppose $d\geq 1$ and let $0<\delta<1$.  By induction,
\begin{align*}
\nfact{d}{s+\delta}-\nfact{d}{s}&=\int_{s-1}^{s}\nfact{d-1}{t+\delta}-\nfact{d-1}{t}\mathrm{d}t\leq \int_{s-1}^s\delta\;\mathrm{d}t=\delta.
\end{align*}

(\ref{item right tail}) This statement is true for $d=0$ by inspection.  Assume that $d\geq 1$ and $\nfact{d-1}{s}=1$ for $s\geq  d-1$.  Then for $s\geq d$, we have that
\[\nfact{d}{s}=\int_{s-1}^s\nfact{d-1}{t}\mathrm{d}t=\int_{s-1}^s1\;\mathrm{d}t=1\]
and the result follows by induction.

(\ref{item first interval behavior})  This is clear from the definition.
\end{proof}

Many properties of the $\rmuM{s}{I}{J}{M}$ immediately imply similar properties for the $s$-multiplicity.  Some of these properties are listed in the next three corollaries.  The first corollary makes explicit the interpolating properties of the $s$-multiplicity, while the second contains some auxiliary results listed for completeness.  The third is the Associativity Formula for $s$-multiplicity.

\begin{corollary} \label{smult interpolation} \standardassumptionsnosdim
\begin{enumerate}[(i)]
\item \label{low s behavior} If $0< s<\min\{1,\sthresh{I}{J}\}$, then $\smuM{s}{I}{J}{M}=e(I;M)$. 
\item \label{high s behavior} If $s>\max\{d,\fthresh{I}{J}\}$, then $\smuM{s}{I}{J}{M}=e_{HK}(J;M)$. 
\item \label{regular behavior} If $R$ is a regular ring, then $\smuu{s}{R}=1$.
\end{enumerate}
\end{corollary}
\begin{proof} Statements (\ref{low s behavior}) and (\ref{high s behavior}) simply combine Lemma \ref{lemma -- thresholds vs rmult} and Lemma \ref{lemma -- properties of E_d(s)}.  For statement (\ref{regular behavior}), we may assume without loss of generality that $R$ is complete with residue field $k$, in which case $R\cong k[[x_1,\ldots, x_d]]$.  The result then follows from Definition \ref{definition -- s-multiplicity} and Proposition \ref{H_s(d) combinatorial}.\end{proof}

\begin{corollary} \label{smult basics} \standardassumptionsnosdim The following statements hold.
\begin{enumerate}[(i)]
\item \label{smult continuous} $\smuM{s}{I}{J}{M}$ is a Lipschitz continuous function of $s$.
\item \label{smult HK compare} $\smuM{s}{I}{J}{M}\leq {e_{HK}(J;M)}/{\nfact{d}{s}}$.
\item \label{smult low dim} If $\dim M<d$ then $\smuM{s}{I}{J}{M}=0$.
\item \label{smult ideals decreasing} If $I'$ and $J'$ are $\m$-primary ideals of $R$ such that $I\subseteq I'$ and $J\subseteq J'$, then $\smuM{s}{I'}{J'}{M}\leq \smuM{s}{I}{J}{M}$. 
\item \label{smult integral closure} If $I'$ is an $\m$-primary ideal of $R$ with the same integral closure as $I$, then $\smuM{s}{I'}{J}{M}=\smuM{s}{I}{J}{M}$.
\item \label{smult tight closure} If $J'$ is an $\m$-primary ideal of $R$ with the same tight closure as $J$, then $\smuM{s}{I}{J'}{M}=\smuM{s}{I}{J}{M}$.
\item \label{smult additivity} If $0\to M'\to M\to M''\to 0$ is a short exact sequence of finitely generated $R$-modules, then  $\smuM{s}{I}{J}{M}=\smuM{s}{I}{J}{M'}+\smuM{s}{I}{J}{M''}$.
\end{enumerate}
\end{corollary}

\begin{proof} (\ref{smult continuous}) We have that $\smuM{s}{I}{J}{M}$ is constant, hence Lipschitz continuous, on $(0,\min\{1,\sthresh{I}{J}\}]$.  By Lemma \ref{lemma -- properties of E_d(s)}, $\nfact{d}{s}$ is Lipschitz continuous and nonzero on $[\min\{1,\sthresh{I}{J}\},\infty)$ and by Theorem \ref{theorem -- rmult lipschitz}, $\rmuM{s}{I}{J}{M}$ is Lipschitz continuous, and so $\smuM{s}{I}{J}{M}$ is Lipschitz continuous on $[\min\{1,\sthresh{I}{J}\},\infty)$.  Thus $\smuM{s}{I}{J}{M}$ is Lipschitz continuous.

 Parts (\ref{smult HK compare})-(\ref{smult tight closure}) follow from Proposition \ref{rmult basics}.  Part (\ref{smult additivity}) follows from Theorem \ref{rmult additivity}.
\end{proof}

\begin{corollary}[Associativity Formula for $s$-Multiplicity] \label{smult AF} \standardassumptionsnos
We have that
\[\smuRM{R}{s}{I}{J}{M}=\sum_{\p\in \Assh R}\smuR{R/\p}{s}{I(R/\p)}{J(R/\p)}\lengthR{R_\p}{M_\p}\]
where $\Assh R=\set{\p\in \Spec R}{\dim R/\p=\dim R}$.
\end{corollary}

\begin{proof}  For any $\p\in \Assh R$, $\dim R/\p=d$, and so 
\[\smuR{R/\p}{s}{I(R/\p)}{J(R/\p)}=\frac{\rmuR{R/\p}{s}{I(R/\p)}{J(R/\p)}}{\nfact{d}{s}}.\]
By Theorem \ref{rmult Associativity Formula}, we have that
\[\rmuRM{R}{s}{I}{J}{M}=\sum_{\p\in \Assh R}\rmuR{R/\p}{s}{I(R/\p)}{J(R/\p)}\lengthR{R_\p}{M_\p}.\]
Therefore, dividing each term of this equation by $\nfact{d}{s}$ proves the result.
\end{proof}

An immediate application of Corollary \ref{smult AF} is the following result, which shows that the $s$-multiplicity of a module is in many cases determined by the $s$-multiplicity of the ring itself.

\begin{proposition} Let $(R,\m)$ be a local domain of characteristic $p>0$ and let $I$ and $J$ be $\m$-primary ideals of $R$.  If $M$ is a finitely generated $R$-module, then $\smuM{s}{I}{J}{M}=  \smu{s}{I}{J}\cdot \rank M$.
\end{proposition}

\begin{proof} By the Associativity Formula, we have that
\begin{align*}
\smuRM{R}{s}{I}{J}{M} &=\smuR{R}{s}{I}{J}\lengthR{R_{(0)}}{M_{(0)}}=\smuR{R}{s}{I}{J}\cdot \rank M.\qedhere\end{align*}
\end{proof}

The problem of finding general bounds for the value of the $s$-multiplicity seems to be difficult, but we have a few results along those lines.  

\begin{proposition} \label{smult ring morphism} Let $\varphi:(R,\m)\to (S,\n)$ be a  local homomorphism of local rings of dimension $d$ and characteristic $p>0$ such that $\m S$ is $\n$-primary, let $I$ and $J$ be $\m$-primary ideals of $R$, and let $M$ be a finitely generated $R$-module.  Then 
\[\smuRM{S}{s}{IS}{JS}{M\otimes_R S}\leq \smuRM{R}{s}{I}{J}{M}\cdot \lengthR{S}{S/\m S}\]
and we have equality if $\varphi$ is a flat ring homomorphism.
\end{proposition}

\begin{proof} For any $R$-module $N$ of finite length, we have that 
\[\lengthR{S}{N\otimes_R S}\leq \lengthR{R}{N}\cdot\lengthR{S}{S/\m S}.\] Thus, for any $s>0$ and $e\in\N$ we have that
\begin{align*}
\lengthR{S}{\frac{M\otimes_R S}{(\iii{s}{p^e}{(IS)}{(JS)})(M\otimes_R S)}}&=\lengthR{S}{\frac{M}{(\ii{s}{p^e})M}\otimes_R S}\\&\leq \lengthR{R}{\frac{M}{(\ii{s}{p^e})M}}\cdot \lengthR{S}{S/\m S}.\end{align*}
Dividing both sides by $p^{ed}$ and taking the limit as $e$ goes to infinity gives us that 
\[\rmuRM{S}{s}{IS}{JS}{M\otimes_R S}\leq \rmuRM{R}{s}{I}{J}{M}\cdot \lengthR{S}{S/\m S},\]
and dividing both sides by $\nfact{d}{s}$ gives us the result for $s$-multiplicity.

If $\varphi$ is a flat ring homomorphism, then for any $R$-module $N$ we have that $\lengthR{S}{N\otimes_R S}= \lengthR{R}{N}\cdot \lengthR{S}{S/\m S}$ and so we have equality everywhere.
\end{proof}

\begin{corollary} If $(R,\m, k)$ be a local ring of characteristic $p>0$ and $I$ is an ideal generated by a system of parameters in $R$, then $ \smuu{s}{I}\leq \length{R/I}$.  Furthermore, equality holds if $R$ is Cohen-Macaulay.
\end{corollary}

\begin{proof} We may assume that $R$ is complete.  Let $d=\dim R$, let $x_1,\ldots, x_d$ be a system of parameters generating $I$, and let $S=k[[x_1,\ldots, x_d]]\subseteq R$.  Now by Proposition \ref{smult ring morphism} and Corollary \ref{smult interpolation}(\ref{regular behavior}), 
\[\smuuR{R}{s}{I} \leq \smuuR{S}{s}{(x_1,\ldots, x_d)}\lengthR{R}{R/I} =\lengthR{R}{R/I}.\]
Furthermore, if $R$ is Cohen-Macaulay, then $R$ is a free $S$-module, hence is flat over $S$, so equality holds.
\end{proof}

\section{$s$-Closure}\label{closure operators section}

The $s$-multiplicity is related to closures, just as the Hilbert-Samuel and Hilbert-Kunz multiplicities are.  We see this already in the guise of Proposition \ref{rmult basics} and Corollary \ref{smult basics} with respect to integral and tight closure.  The natural question to ask at this point is whether there are closures that are similarly related to the various $s$-multiplicities.  In this section we define these closures and show that in sufficiently nice rings, we get a strong connection between the closure operators and the $s$-multiplicity.  We use the notation $R^\circ$ to stand for the complement of the union of the minimal primes of $R$.

\begin{definition} Let $R$ be a ring of characteristic $p>0$, let $I$ be an ideal of $R$, and let $s\geq 1$ be a real number.  An element $x\in R$ is said to be in the \emph{weak $s$-closure of $I$} if there exists $c\in R^\circ$ such that for all $e\gg 0$,
$cx^{p^e}\in I^{\lceil sp^e\rceil}+I^{[p^e]}$.
We denote the set of all $x$ in the weak $s$-closure of $I$ by $\wsc{s}{I}$.
\end{definition}

\begin{remark} If $I$ is of positive height, then $x\in \wsc{s}{I}$ if and only if there exists $c\in R^\circ$ such that $cx^{p^e}\in I^{\lceil sp^e\rceil}+I^{[p^e]}$ for all $e\in\N$.  To see this, suppose that there exists $c'\in R^\circ$ and $e'\in \N$ such that $c'x^{p^e}\in  I^{\lceil sp^e\rceil}+I^{[p^e]}$ for $e> e'$.  Since $I$ is of positive height, there exists $c''\in (I^{\lceil sp^{e'}\rceil}+I^{[p^{e'}]})\cap R^\circ$.  Setting $c=c'c''$, we have that $c\in R^\circ$ and $cx^{p^e}\in\iii{s}{p^e}{I}{I}$ for all $e\in \N$.
\end{remark}

For a given ideal $I$, $\wsc{s}{I}$ is clearly an ideal containing $I$.  However, it is not clear that the weak $s$-closure is idempotent; that is, it is not clear that $\wsc{s}{(\wsc{s}{I})}=\wsc{s}{I}$.  If the ring is noetherian, we can construct an idempotent operation out of the weak $s$-closure by iterating the operation until the chain of ideals stabilizes.

\begin{definition}Let $R$ be a ring of characteristic $p>0$, let $I$ be an ideal of $R$, and let $s\geq 1$ be a real number.  The $s$-closure of $I$ is defined to be the union of the following chain of ideals:
\[I\subseteq \wsc{s}{I}\subseteq \wsc{s}{(\wsc{s}{I})}\subseteq \wsc{s}{\left(\wsc{s}{(\wsc{s}{I})}\right)}\subseteq \cdots.\]
We denote this ideal by $\scl{s}{I}$.
\end{definition}

Notice that, for $s=1$, the $s$-closure is integral closure, and for $s>\fthresh{I}{I}$, the $s$-closure is tight closure.  Furthermore, if $s\leq s'$, then $\scl{s}{I}\supseteq \scl{s'}{I}$ for all ideals $I$.  Thus the $s$-closure interpolates monotonically between integral closure and tight closure as $s$ increases.  One should note that new closures do in fact arise:

\begin{example} \label{basic sclosure example} Let $R=k[[x,y]]$, where $k$ is a field of characteristic $p>0$.  Let $I=(x^3,y^3)$.  Then
\[\scl{s}{I}=\begin{cases} (x,y)^3 & \text{if }s=1 \\ (x^3,x^2y^2,y^3) & \text{if }1<s\leq \frac{4}{3} \\ (x^3,y^3) & \text{if }s>\frac{4}{3}.\end{cases}\]
In particular, if $1< s\leq \frac{4}{3}$, then $I=I^\ast \subsetneq \scl{s}{I} \subsetneq \overline{I}=(x,y)^3$.
\end{example}

Example \ref{basic sclosure example} demonstrates that in some cases, an ideal $I$ will only have finitely many distinct $s$-closures for various values of $s$; in fact, this will occur whenever $R$ is local and $I$ is primary to the maximal ideal.  However, even in regular rings there can be infinitely many distinct $s$-closures.

\begin{example} \label{expanded sclosure example} Let $R=k[[x,y]]$, where $k$ is a field of characteristic $p>0$.  Let $1\leq s < s'\leq 2$.  Choose $n\in\N$ such that $n> 2/(s'-s)$, and let $I=(x^{2n},y^{2n})$.  Then $x^{\lceil sn\rceil }y^{\lceil sn\rceil} \in \wsc{s}{I}$, since for any $e\in \N$, 
\[2\left\lfloor \frac{2n+\lceil sn\rceil p^e}{2n}\right\rfloor \geq2\left\lfloor1+\frac{s}{2}p^e\right\rfloor \geq sp^e,\] and so $x^{2n}y^{2n}(x^{\lceil sn\rceil}y^{\lceil sn\rceil})^{p^e}\in (x^{2n},y^{2n})^{\lceil sp^e\rceil}$.  However, $x^{\lceil sn\rceil }y^{\lceil sn\rceil} \notin\wsc{s'}{I}$, since for any $a\in \N$, letting $e\in \N$ such that $p^e\geq a$, we have that 
\[2\left\lfloor\frac{a+\lceil sn\rceil p^e}{2n}\right\rfloor\leq \frac{a+(sn+1)p^e}{n}\leq \frac{(sn+2)p^e}{n}=sp^e+\frac{2p^e}{n} <sp^e+(s'-s)p^e=s'p^e\]
 and so $x^ay^a(x^{\lceil sn\rceil}y^{\lceil sn\rceil})^{p^e}\notin (x^{2n},y^{2n})^{\lceil s'p^e\rceil}$.  Thus $\wsc{s}{I}\neq \wsc{s'}{I}$, and hence $\scl{s}{I}\neq \scl{s'}{I}$ by Theorem \ref{sclosure-smultiplicity}.  Thus we find that there are infinitely many distinct $s$-closures on $R$, one for every real number in the interval $[1,2]$.
\end{example}

If $I$ and $I'$ have the same integral closure, then $e(I)=e(I')$, while if $I$ and $I'$ have the same tight closure, then $e_{HK}(I)=e_{HK}(I')$.  Our main theorem in this section is a similar result for $s$-multiplicity and $s$-closure.

\begin{theorem} \label{sclosure-smultiplicity} Let $(R,\m)$ be a local ring of characteristic $p>0$ and let $I$ and $J$ be $\m$-primary ideals of $R$ with $I\subseteq J$.  If $J \subseteq \scl{s}{I}$, then $\smuu{s}{J}=\smuu{s}{I}$.  If $R$ is an $F$-finite complete domain, then the converse holds and $\scl{s}{I}=\wsc{s}{I}$.
\end{theorem}

\begin{proof} Let $d=\dim R$.  Suppose that $x\in \wsc{s}{I}$, so that there exists $c\in R^\circ$  such that for all $e\gg 0$, we have that  $cx^{p^e}\in  I^{\lceil sp^e\rceil}+I^{[p^e]}\subseteq I^{p^e}$.  Hence  $x$ is in the integral closure of $I$ and so $\rmu{s}{(I,x)}{(I,x)}=\rmu{s}{I}{(I,x)}$ by Proposition \ref{rmult basics}(\ref{rmult integral closure}).  Now for large $e\in\N$, $c$ annihilates $\frac{I^{\lceil sp^e\rceil}+(I,x)^{[p^e]}}{I^{\lceil sp^e\rceil}+I^{[p^e]}}$.  Let $S=R/cR$, so that for $e\gg 0$,
 \begin{align*}
\lengthR{R}{\frac{I^{\lceil sp^e\rceil}+(I,x)^{[p^e]}}{I^{\lceil sp^e\rceil}+I^{[p^e]}}}
= \lengthR{S}{\frac{I^{\lceil sp^e\rceil}+(I,x)^{[p^e]}}{I^{\lceil sp^e\rceil}+I^{[p^e]}}\otimes S}
=\lengthR{S}{\frac{(IS)^{\lceil sp^e\rceil}+\left((I,x)S\right)^{[p^e]}}{(IS)^{\lceil sp^e\rceil}+(IS)^{[p^e]}}}
\end{align*}
 And so, since $\dim S\leq d-1$,
\begin{align*}
\rmu{s}{I}{I}-\rmu{s}{I}{(I,x)}
=&\lim_{e\to\infty}\frac{1}{p^{ed}}\lengthR{R}{\frac{I^{\lceil sp^e\rceil}+(I,x)^{[p^e]}}{I^{\lceil sp^e\rceil}+I^{[p^e]}}}\\
=&\lim_{e\to\infty}\frac{1}{p^{ed}}\lengthR{S}{\frac{(IS)^{\lceil sp^e\rceil}+\left((I,x)S\right)^{[p^e]}}{(IS)^{\lceil sp^e\rceil}+(IS)^{[p^e]}}}\\
=&\left(\lim_{e\to\infty}\frac{1}{p^{e(d-\dim S)}}\right)(\rmuR{S}{s}{IS}{IS}-\rmuR{S}{s}{IS}{(I,x)S})
=0.
\end{align*}
Therefore $\rmuu{s}{(I,x)}=\rmuu{s}{I}$ for any $x\in \wsc{s}{I}$, hence $\rmuu{s}{\wsc{s}{I}}=\rmuu{s}{I}$.  By induction, $\rmuu{s}{\scl{s}{I}}=\rmuu{s}{I}$, hence $\rmuu{s}{J}=\rmuu{s}{I}$ and so $\smuu{s}{J}=\smuu{s}{I}$.

Now suppose that $R$ is an $F$-finite complete domain and $x\in R$ such that $\smuu{s}{(I,x)}=\smuu{s}{I}$.  In this case $\rmuu{s}{(I,x)}=\rmuu{s}{I}$, and so $\rmu{s}{I}{(I,x)}=\rmu{s}{I}{I}$, and therefore
\begin{align*}
0=&\lim_{e\to\infty}\frac{1}{p^{ed}}\length{\frac{I^{\lceil sp^e\rceil}+(I,x)^{[p^e]}}{I^{\lceil sp^e\rceil}+I^{[p^e]}}}
=\lim_{e\to\infty}\frac{1}{p^{ed}}\length{R/\left(I^{\lceil sp^e\rceil}+I^{[p^e]}:_R x^{p^e}\right)}.
\end{align*}
Let $\psi:F_\ast R\to R$ be a nonzero $p^{-1}$-linear map and let 
$\varphi(-)=\psi\left(F_\ast(f_1^{p-1}\cdots f_n^{p-1})\cdot -\right)$, where $f_1,\ldots, f_n$ is a generating set for $I$.  Then
\begin{align*}
\varphi\left(F_\ast\left(I^{\lceil sp^{e+1}\rceil}+I^{[p^{e+1}]}:_R x^{p^{e+1}}
\right)\right)\cdot x^{p^e}&
\subseteq \varphi\left(F_\ast\left(I^{\lceil sp^{e+1}\rceil}+I^{[p^{e+1}]}\right)\right)\\&\subseteq  \psi\left(F_\ast\left(f_1^{p-1}\cdots f_n^{p-1}I^{\lceil sp^{e+1}\rceil}\right)\right)+I^{[p^e]}
\end{align*}
Let $a_1,\cdots,a_n\in \N$ with $a_1+\cdots +a_n\geq sp^{e+1}$.  Then
\[
\sum_{i=1}^n \left\lfloor\frac{a_i+p-1}{p}\right\rfloor
\geq \sum_{i=1}^n\frac{a_i}{p}\geq sp^e\]
and so $f_1^{p-1}\cdots f_n^{p-1}I^{\lceil sp^{e+1}\rceil}\subseteq \left(I^{\lceil sp^e\rceil }\right)^{[p]}$.  Therefore 
$\psi\left(F_\ast\left(f_1^{p-1}\cdots f_n^{p-1}I^{\lceil sp^{e+1}\rceil}\right)\right)\subseteq I^{\lceil sp^e\rceil }$ and so
\[\varphi\left(F_\ast\left(I^{\lceil sp^{e+1}\rceil}+I^{[p^{e+1}]}:_R x^{p^{e+1}}
\right)\right)\cdot x^{p^e}\subseteq I^{\lceil sp^{e}\rceil}+I^{[p^{e}]},\] that is,
\[\varphi\left(F_\ast\left(I^{\lceil sp^{e+1}\rceil}+I^{[p^{e+1}]}:_R x^{p^{e+1}}
\right)\right)\subseteq \left( I^{\lceil sp^{e}\rceil}+I^{[p^{e}]}:_R x^{p^e}\right).\]
Since this holds for all $e\in\N$, by \cite[Theorem 5.5]{PolstraTucker-FsignatureHilbertKunzMultiplicity}, we must have that
$\bigcap_{e\geq 0} \left( I^{\lceil sp^{e}\rceil}+I^{[p^{e}]}:_R x^{p^e}\right) \neq 0$,
that is, there is some $0\neq c\in R$ such that for all $e\in\N$, $cx^{p^e}\subseteq I^{\lceil sp^{e}\rceil}+I^{[p^{e}]}$.  Therefore $x\in \wsc{s}{I}$.

Thus we have that if $R$ is an $F$-finite complete domain and $\rmuu{s}{(I,x)}=\rmuu{s}{I}$, then $x\in \wsc{s}{I}$.  Therefore if $\rmuu{s}{J}=\rmuu{s}{I}$ then $J\subseteq \wsc{s}{I}\subseteq \scl{s}{I}$.  Furthermore, in this case, if $x\in \scl{s}{I}$, then $\rmuu{s}{(I,x)}=\rmuu{s}{I}$ and hence $x\in \wsc{s}{I}$.  Therefore $\scl{s}{I}=\wsc{s}{I}$.
\end{proof}

\section{$s$-Multiplicity of Toric Rings}\label{smult toric rings section}

In this section we construct an equivalence between $s$-multiplicity for toric rings and certain volumes in Euclidean space.  We will then use that equivalence to compute the $s$-multiplicity for a few toric rings.  See \cite{HernandezJeffries-LimitsPrimeCharacteristic} for a more general treatment of the correspondence between limits in positive characteristic and volumes in real space.

\begin{definition} Let $k$ be a field.  By a \emph{normal toric ring of dimension $d$ over $k$}, or simply \emph{toric ring}, we will mean the ring $k[[S]]$, where $S={\sigma^{\!\vee}}\cap \Z^d$, ${\sigma^{\!\vee}}$ is a cone in $\R^d$ not containing any line through the origin, and $S$ inherits the semigroup structure of $\Z^d$.  Furthermore, we will require that the cone ${\sigma^{\!\vee}}$ be rational, that is, ${\sigma^{\!\vee}}=\cone(v_1,\ldots, v_n)$ for some $v_1,\ldots, v_n\in \Z^d$, and of full dimension, that is, the $\R$-span of ${\sigma^{\!\vee}}$ is all of $\R^d$.  We will denote the monomial elements of $k[[S]]$ by $x^v$ for $v\in S$, and if ${\sigma^{\!\vee}}=\cone(v_1,\ldots,v_n)$, we may write $k[[x^{v_1},\ldots, x^{v_n}]]$ for $k[[S]]$.
\end{definition}

\begin{definition} For a monomial ideal $I\subseteq k[[S]]$, where $k[[S]]$ is a toric ring, we denote by $\Exp I$ the set $\set{v\in S}{x^v\in I}$ and by $\Hull I$ the convex hull of $\Exp I$ in $\R^d$.  
\end{definition}

\begin{lemma} \label{asymptotic Exp} Let $(R,\m)=(k[[S]],(S))$ be a normal toric ring of dimension $d$ over a field $k$ of characteristic $p>0$, where $S={\sigma^{\!\vee}}\cap \Z^d$, let $I=(x^{u_1},\ldots,x^{u_n})$ be a monomial ideal of $R$.  For any $m,e\in \N$ with $m\geq 1$,
\[(m+n) \Hull I\cap \Z^d\subseteq \Exp I^m\subseteq  m \Hull I\cap \Z^d\quad \text{and}\quad 
\Exp I^{[p^e]}=\left(p^e\Exp I+\sigma^\vee\right)\cap \Z^d\]
\end{lemma}

\begin{proof} Let $v\in (m+n) \Hull I\cap \Z^d$.  Then there exist $a_1,\ldots, a_n\in \R_{\geq 0}$ such that $a_1+\cdots+a_n=1$ and $v\in (m+n)(a_1u_1+\cdots+a_nu_n+{\sigma^{\!\vee}})\cap \Z^d$.  For each $1\leq i\leq n$, let $b_i=\lfloor (m+n)a_i\rfloor$.  Since each $u_i\in {\sigma^{\!\vee}}$, we have that
\[v\in (m+n)(a_1u_1+\cdots +a_nu_n+{\sigma^{\!\vee}})\cap \Z^d\subseteq (b_1u_1+\cdots+b_nu_n+{\sigma^{\!\vee}})\cap \Z^d.\]
Since $b_1+\cdots +b_n\geq (m+n)(a_1+\cdots +a_n)-n=m$, we have that $v\in \Exp I^m$.  This shows the first inclusion in the first statement.

A monomial $x^v$ is in $I^m$ if and only if $v\in (a_1u_1+\cdots+a_nu_n+{\sigma^{\!\vee}})\cap \Z^d$ for some $a_1,\ldots, a_n\in \N$ with $a_1+\cdots +a_n= m$.  If this is the case then 
\[v\in \left(m\left(\frac{a_1}{m}u_1+\cdots+\frac{a_n}{m}u_n\right)+{\sigma^{\!\vee}}\right)\cap \Z^d\subseteq  m\Hull I\cap \Z^d.\]
This shows the second inclusion in the first statement.

A monomial $x^v$ is in $I^{[p^e]}$ if and only if $v\in p^eu+{\sigma^{\!\vee}}$ for some $u\in\Exp I$.  That is, $x^v\in I^{[p^e]}$ if and only if
\[v\in \bigcup_{u\in\Exp I}(p^eu+{\sigma^{\!\vee}})\cap \Z^d=\left(p^e\bigcup_{u\in \Exp I}(u+{\sigma^{\!\vee}})\right)\cap \Z^d=(p^e\Exp I+\sigma^\vee)\cap\Z^d,\]
which proves the second statement.
\end{proof}

\begin{theorem}\label{rmult toric ring} Let $(R,\m)=(k[[S]],(S))$ be a normal toric ring of dimension $d$ over a field $k$ of characteristic $p>0$, where $S={\sigma^{\!\vee}}\cap \Z^d$, and let $I$ and $J$ be $\m$-primary monomial ideals of $R$.  Then 
\[\rmu{s}{I}{J}=\vol\left({\sigma^{\!\vee}} \setminus \left(s \Hull I \cup (\Exp J+\sigma^{\! \vee}\right)\right)\]
where $\vol(-)$ is the standard Euclidean volume in $\R^d$.
\end{theorem}

\begin{proof} Let $e\in \N$.  The length of $R/(\ii{s}{p^e})$ is precisely the size of the set 
\[V_e=\set{v\in S}{x^v\notin \ii{s}{p^e}}=\set{v\in S}{v\notin \Exp I^{\lceil sp^e\rceil}\cup \Exp J^{[p^e]}}.\]
From Lemma \ref{asymptotic Exp}, we have that 
\begin{align*} \left({\sigma^{\!\vee}}\setminus \left(sp^e \Hull I\cup p^e\Exp J+{\sigma^{\!\vee}}\right)\right)\cap \Z^d  \subseteq V_e \subseteq \left({\sigma^{\!\vee}} \setminus \left((sp^e+n) \Hull I\cup p^e\Exp J+{\sigma^{\!\vee}}\right)\right) \cap \Z^d.\end{align*}
Scaling every set by $\frac{1}{p^e}$ in each dimension, we get that
\begin{align*}\left({\sigma^{\!\vee}}\setminus \left(s\Hull I\cup \Exp J+{\sigma^{\!\vee}}\right)\right)\cap \left(\!\frac{1}{p^e}\Z\!\right)^d&\subseteq \frac{1}{p^e}V_e\subseteq \left({\sigma^{\!\vee}} \setminus \left((s+n/p^e) \Hull I\cup \Exp J + {\sigma^{\!\vee}}\right)\right) \cap \left(\!\frac{1}{p^e}\Z\!\right)^d.\end{align*}
Since the volume of ${\sigma^{\!\vee}}\setminus \left(s\Hull I\cup \Exp J+{\sigma^{\!\vee}}\right)$ is equal to the volume of its interior, we obtain that
\begin{align*}
& \vol\left({\sigma^{\!\vee}}\setminus \left(s\Hull I\cup \Exp J+{\sigma^{\!\vee}}\right)\right)\\
=&\lim_{e\to\infty}\frac{1}{p^{ed}} \left|\left({\sigma^{\!\vee}}\setminus \left(s\Hull I\cup \Exp J+{\sigma^{\!\vee}}\right)\right)\cap \left(\frac{1}{p^e}\Z\right)^d\right|\\
\leq&  \lim_{e\to\infty}\frac{1}{p^{ed}}|V_e|\\
\leq &\lim_{e\to\infty}\frac{1}{p^{ed}} \left|\left({\sigma^{\!\vee}}\setminus \left((s+n/p^e)\Hull I\cup \Exp J+{\sigma^{\!\vee}}\right)\right)\cap \left(\frac{1}{p^e}\Z\right)^d\right|\\
=&\vol\left({\sigma^{\!\vee}}\setminus \left(s\Hull I\cup \Exp J+{\sigma^{\!\vee}}\right)\right)
\end{align*}
And so we have equality throughout.  Since $\displaystyle{\rmu{s}{I}{J}=\lim_{e\to\infty}\frac{1}{p^{ed}}|V_e|}$, the theorem is proved.
\end{proof}

Theorem \ref{rmult toric ring} allows us to calculate the $s$-multiplicity of toric rings.  We compute two examples.

\begin{example}[$A_n$ Singularities] Let $n\in \N$, $n\geq 1$, and take  
\[A_n=k[[x_1,x_2,x_3]]/(x_1x_2-x_3^{n+1})\cong k[[x,y,x^{-1}y^{n+1}]].\]
The geometry of this toric ring is illustrated below.  The shaded region corresponds to the cone ${\sigma^{\!\vee}}$, and the lattice points $(1,0)$, $(0,1)$, and $(-1,n+1)$ correspond to $x$, $y$, and $x^{-1}y^{n+1}$, respectively.

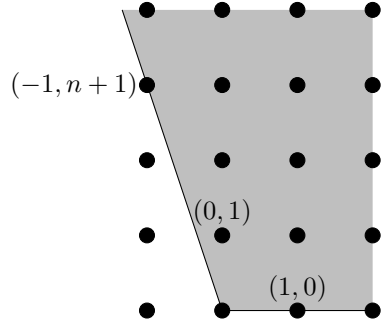
\begin{center}
\begin{tikzpicture}[scale=1]
\path[fill=lightgray](0,0)--(-1.33,4)--(2,4)--(2,0);
\foreach \x in {-1,0,...,2}{
\foreach \y in {0,1,...,4}{
  \node[draw,circle,fill,inner sep=2] at (\x,\y){};
 }
}
\draw (0,0)--(-1.33,4);
\draw (0,0)--(2,0);
\node [above] at (1,0) {$(1,0)$};
\node [above] at (0,1) {$(0,1)$};
\node [left] at (-1,3) {$(-1,n+1)$};
\end{tikzpicture}
\end{center}
We wish to calculate $\smuu{s}{A_n}$, so we need to calculate $\Hull \m$ and $\Exp \m + \sigma^{\!\vee}$ where $\m=(x,y,x^{-1}y^{n+1})$.  These are illustrated below.
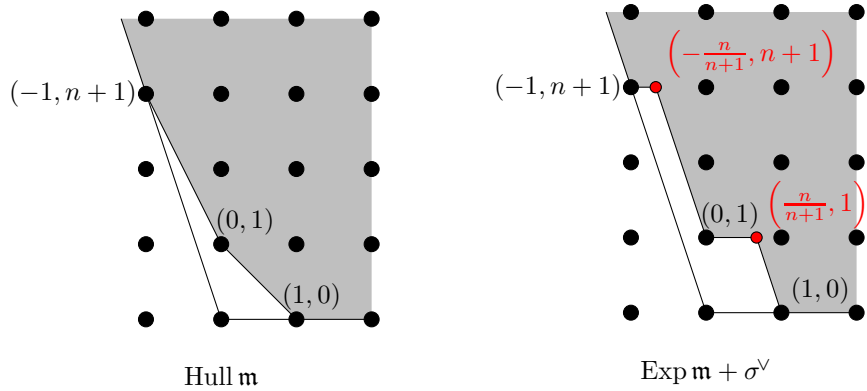
\begin{center}
\begin{align*}
\begin{tikzpicture}[scale=1]
\path[fill=lightgray](1,0)--(0,1)--(-1,3)--(-1.33,4)--(2,4)--(2,0);
\foreach \x in {-1,0,...,2}{
\foreach \y in {0,1,...,4}{
  \node[draw,circle,fill,inner sep=2] at (\x,\y){};
 }
}
\draw (2,0)--(0,0)--(-1.33,4);
\draw (1,0)--(0,1)--(-1,3);
\node [above] at (1.2,0) {$(1,0)$};
\node [above right] at (-0.2,1) {$(0,1)$};
\node [left] at (-1,3) {$(-1,n+1)$};
\node [below] at (0,-0.5) {$\Hull \m$};
\end{tikzpicture}
&\qquad\qquad
\begin{tikzpicture}[scale=1]
\path[fill=lightgray](1,0)--(0.67,1)--(0,1)--(-0.67,3)--(-1,3)--(-1.33,4)--(2,4)--(2,0);
\foreach \x in {-1,0,...,2}{
\foreach \y in {0,1,...,4}{
  \node[draw,circle,fill,inner sep=2] at (\x,\y){};
 }
}
\draw (2,0)--(0,0)--(-1.33,4);
\draw (1,0)--(0.67,1)--(0,1)--(-0.67,3)--(-1,3);
\node [above right] at (1,0) {$(1,0)$};
\node [above right] at (-0.2,1) {$(0,1)$};
\node [left] at (-1,3) {$(-1,n+1)$};
\node [draw,circle,fill=red,inner sep=1.5] at (0.67,1) {};
\node [above right,color=red] at (0.67,1) {$\left(\frac{n}{n+1},1\right)$};
\node [draw,circle,fill=red,inner sep=1.5] at (-0.67,3) {};
\node [above right,color=red] at (-0.67,3) {$\left(-\frac{n}{n+1},n+1\right)$};
\node [below] at (0,-0.5) {$\Exp \m + \sigma^{\!\vee}$};
\end{tikzpicture}
\end{align*}
\end{center}

There are three situations to consider: $s\leq 1$, $1\leq s\leq2-\frac{1}{n+1}$, and $s\geq 2-\frac{1}{n+1}$.  When $s\leq 1$, $s \Hull \m\cup\Exp \m + \sigma^{\!\vee}$ is illustrated below:
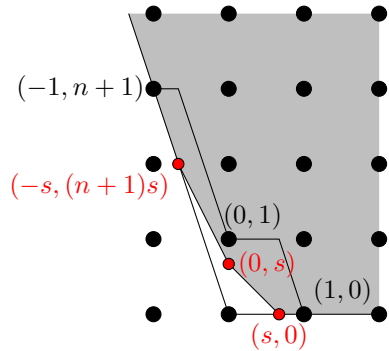
\begin{center}
\begin{tikzpicture}[scale=1]
\path[fill=lightgray](0.67,0)--(0,0.67)--(-0.67,2)--(-1.33,4)--(2,4)--(2,0);
\foreach \x in {-1,0,...,2}{
\foreach \y in {0,1,...,4}{
  \node[draw,circle,fill,inner sep=2] at (\x,\y){};
 }
}
\draw (0,0)--(-1.33,4);
\draw (0,0)--(2,0);
\draw (0.67,0)--(0,0.67);
\draw (0,0.67)--(-0.67,2);
\draw (1,0)--(0.67,1)--(0,1)--(-0.67,3)--(-1,3);
\node [above right] at (1,0) {$(1,0)$};
\node [above right] at (-0.2,1) {$(0,1)$};
\node [left] at (-1,3) {$(-1,n+1)$};
\node [draw,circle,fill=red,inner sep=1.5] at (0.67,0) {};
\node [below,color=red] at (0.67,0) {$(s,0)$};
\node [draw,circle,fill=red,inner sep=1.5] at (0,0.67) {};
\node [right,color=red] at (0,0.67) {$(0,s)$};
\node [draw,circle,fill=red,inner sep=1.5] at (-0.67,2) {};
\node [below left,color=red] at (-0.67,2) {$(-s,(n+1)s)$};
\end{tikzpicture}
\end{center}
From this we can compute $\rmuu{s}{A_n} = s^2$ for $s\leq 1$.

Now suppose that $1\leq s\leq 2-\frac{1}{n+1}$.  The picture now becomes
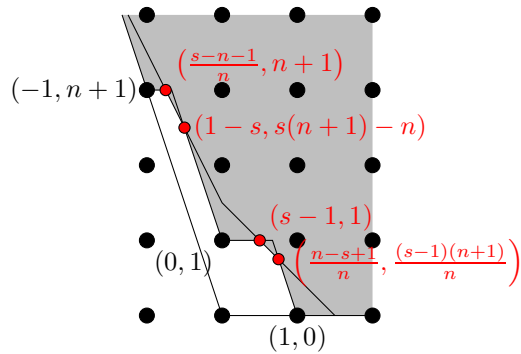
\begin{center}
\begin{tikzpicture}[scale=1]
\path[fill=lightgray](1,0)--(0.75,0.75)--(0.5,1)--(0,1)--(-0.5,2.5)--(-0.75,3)--(-1,3)--(-1.33,4)--(2,4)--(2,0);
\foreach \x in {-1,0,...,2}{
\foreach \y in {0,1,...,4}{
  \node[draw,circle,fill,inner sep=2] at (\x,\y){};
 }
}
\draw (2,0)--(0,0)--(-1.33,4);
\draw (1,0)--(0.67,1)--(0,1)--(-0.67,3)--(-1,3);
\draw (1.5,0)--(0,1.5)--(-1.25,4);
\node [below] at (1,0) {$(1,0)$};
\node [below left] at (0,1) {$(0,1)$};
\node [left] at (-1,3) {$(-1,n+1)$};
\node [draw,circle,fill=red,inner sep=1.5] at (0.75,0.75) {};
\node [right,color=red] at (0.8,0.7) {$\left(\frac{n-s+1}{n},\frac{(s-1)(n+1)}{n}\right)$};
\node [draw,circle,fill=red,inner sep=1.5] at (0.5,1) {};
\node [above right,color=red] at (0.5,1) {$(s-1,1)$};
\node [draw,circle,fill=red,inner sep=1.5] at (-0.5,2.5) {};
\node [right,color=red] at (-0.5,2.5) {$(1-s,s(n+1)-n)$};
\node [draw,circle,fill=red,inner sep=1.5] at (-0.75,3) {};
\node [above right,color=red] at (-0.75,3) {$\left(\frac{s-n-1}{n},n+1\right)$};
\end{tikzpicture}
\end{center}
Calculating the area of the unshaded region in ${\sigma^{\!\vee}}$ gives 
\[\rmuu{s}{A_n}=
-\frac{n+1}{n}(s-1)^2+2(s-1)+1\]
when $1\leq  s\leq 2-\frac{1}{n+1}$.

Now consider the case when $s\geq 2-\frac{1}{n+1}$.  In this case the picture becomes
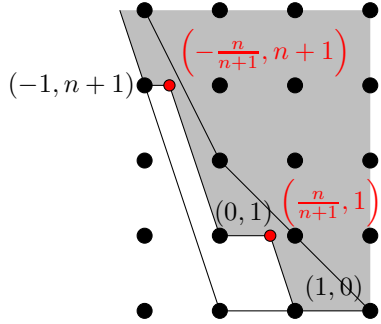
\begin{center}
\begin{tikzpicture}[scale=1]
\path[fill=lightgray](1,0)--(0.67,1)--(0,1)--(-0.67,3)--(-1,3)--(-1.33,4)--(2,4)--(2,0);
\foreach \x in {-1,0,...,2}{
\foreach \y in {0,1,...,4}{
  \node[draw,circle,fill,inner sep=2] at (\x,\y){};
 }
}
\draw (2,0)--(0,0)--(-1.33,4);
\draw (1,0)--(0.67,1)--(0,1)--(-0.67,3)--(-1,3);
\draw (2,0)--(0,2)--(-1,4);
\node [above right] at (1,0) {$(1,0)$};
\node [above right] at (-0.2,1) {$(0,1)$};
\node [left] at (-1,3) {$(-1,n+1)$};
\node [draw,circle,fill=red,inner sep=1.5] at (0.67,1) {};
\node [above right,color=red] at (0.67,1) {$\left(\frac{n}{n+1},1\right)$};
\node [draw,circle,fill=red,inner sep=1.5] at (-0.67,3) {};
\node [above right,color=red] at (-0.67,3) {$\left(-\frac{n}{n+1},n+1\right)$};
\end{tikzpicture}
\end{center}
And so we compute $\rmuu{s}{A_n}=2-\frac{1}{n+1}$ when $s\geq 2-\frac{1}{n+1}$.

With this, we can write down the $s$-multiplicity for the $A_n$ singularities:
\[\smuu{s}{A_n}=\begin{cases}
2 & \text{if } 0 < s< 1\\
\frac{-\frac{n+1}{n}(s-1)^2+2(s-1)+1}{\frac{1}{2}s^2-(s-1)^2} & \text{if } 1\leq  s < 2-\frac{1}{n+1}\\
\frac{2-\frac{1}{n+1}}{\frac{1}{2}s^2-(s-1)^2} & \text{if } 2-\frac{1}{n+1}\leq s < 2\\
2-\frac{1}{n+1} & \text{if } s \geq 2.
\end{cases}\]
 \end{example}
 
 \begin{example} Let $k$ be a field, and consider the $n$th 2-dimensional Veronese subring $V_n=k[[x,xy,\ldots,xy^{n}]]$.  The geometry of this ring is illustrated below; the shaded region corresponds to $\sigma^{\!\vee}$ and for $0\leq a\leq n$, the lattice points $(1,a)$ corresponds to the monomial $xy^a$.

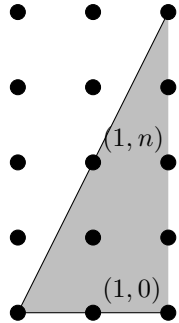
 \begin{center}
\begin{tikzpicture}[scale=1]
\path[fill=lightgray](2,0)--(0,0)--(2,4);
\foreach \x in {0,...,2}{
\foreach \y in {0,1,...,4}{
  \node[draw,circle,fill,inner sep=2] at (\x,\y){};
 }
}
\draw (2,0)--(0,0)--(2,4);
\node[above right] at (1,0) {$(1,0)$};
\node[above right] at (1,2) {$(1,n)$};
\end{tikzpicture}
\end{center}
 
  Letting $\m=(x,xy,\ldots,xy^n)$ we have the following pictures for $\Hull \m$ and $\Exp \m + \sigma^{\!\vee}$.
 
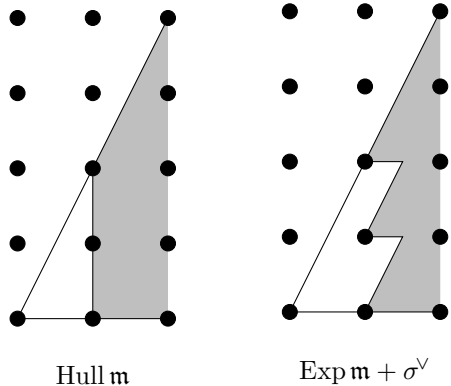
 \begin{center}
\begin{align*}
\begin{tikzpicture}[scale=1]
\path[fill=lightgray](2,0)--(1,0)--(1,2)--(2,4);
\foreach \x in {0,...,2}{
\foreach \y in {0,1,...,4}{
  \node[draw,circle,fill,inner sep=2] at (\x,\y){};
 }
}
\draw (2,0)--(0,0)--(2,4);
\draw (1,0)--(1,2);
\node [below] at (1,-0.5) {$\Hull \m$};
\end{tikzpicture}
&\qquad\qquad
\begin{tikzpicture}[scale=1]
\path[fill=lightgray](2,0)--(1,0)--(1.5,1)--(1,1)--(1.5,2)--(1,2)--(2,4);
\foreach \x in {0,...,2}{
\foreach \y in {0,1,...,4}{
  \node[draw,circle,fill,inner sep=2] at (\x,\y){};
 }
}
\draw (2,0)--(0,0)--(2,4);
\draw (1,0)--(1.5,1)--(1,1)--(1.5,2)--(1,2);
\node [below] at (1,-0.5) {$\Exp \m + \sigma^{\!\vee}$};
\end{tikzpicture}
\end{align*}
\end{center}
 
 Thus we have the following pictures for various values of $s$:
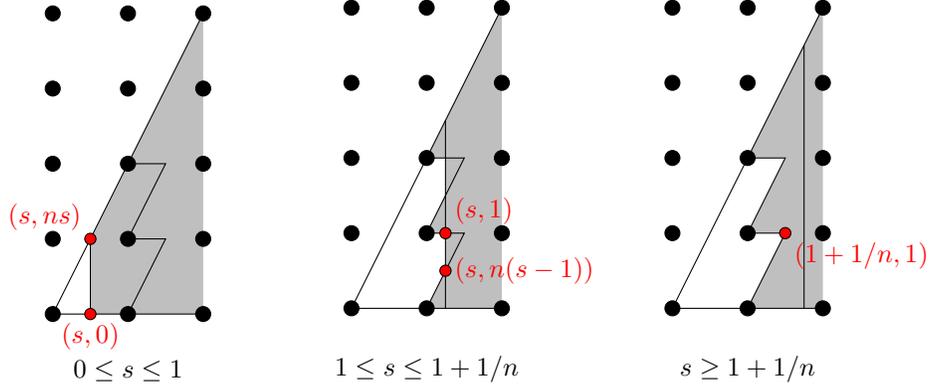
 \begin{center}
 \begin{tikzpicture}[scale=1]
\path[fill=lightgray](2,0)--(0.5,0)--(0.5,1)--(2,4);
\foreach \x in {0,...,2}{
\foreach \y in {0,1,...,4}{
  \node[draw,circle,fill,inner sep=2] at (\x,\y){};
 }
}
\draw (2,0)--(0,0)--(2,4);
\draw (0.5,0)--(0.5,1);
\draw (1,0)--(1.5,1)--(1,1)--(1.5,2)--(1,2);
\node[below,color=red] at (0.5,0) {$(s,0)$};
\node[draw,circle,fill=red,inner sep=1.5] at (0.5,0) {};
\node[above left,color=red] at (0.5,1) {$(s,ns)$};
\node[draw,circle,fill=red,inner sep=1.5] at (0.5,1) {};
\node[below] at (1,-0.5) {$0\leq s\leq 1$};
\end{tikzpicture}
 \qquad\qquad
 \begin{tikzpicture}[scale=1]
\path[fill=lightgray](2,0)--(1,0)--(1.25,0.5)--(1.25,1)--(1,1)--(1.25,1.5)--(1.25,2)--(1,2)--(2,4);
\foreach \x in {0,...,2}{
\foreach \y in {0,1,...,4}{
  \node[draw,circle,fill,inner sep=2] at (\x,\y){};
 }
}
\draw (2,0)--(0,0)--(2,4);
\draw (1,0)--(1.5,1)--(1,1)--(1.5,2)--(1,2);
\draw (1.25,0)--(1.25,2.5);
\node[right,color=red] at (1.25,0.5) {$(s,n(s-1))$};
\node[draw,circle,fill=red,inner sep=1.5] at (1.25,0.5) {};
\node[above right,color=red] at (1.25,1) {$(s,1)$};
\node[draw,circle,fill=red,inner sep=1.5] at (1.25,1) {};
\node [below] at (1,-0.5) {$1\leq s\leq 1+1/n$};
\end{tikzpicture}
 \quad\quad
 \begin{tikzpicture}[scale=1]
\path[fill=lightgray](2,0)--(1,0)--(1.5,1)--(1,1)--(1.5,2)--(1,2)--(2,4);
\foreach \x in {0,...,2}{
\foreach \y in {0,1,...,4}{
  \node[draw,circle,fill,inner sep=2] at (\x,\y){};
 }
}
\draw (2,0)--(0,0)--(2,4);
\draw (1,0)--(1.5,1)--(1,1)--(1.5,2)--(1,2);
\draw (1.75,0)--(1.75,3.5);
\node[below right,color=red] at (1.5,1) {$(1+1/n,1)$};
\node[draw,circle,fill=red,inner sep=1.5] at (1.5,1) {};
\node [below] at (1,-0.5) {$s\geq 1+1/n$};
\end{tikzpicture}
  \end{center}
  With these figures we can caluclate $\rmuu{s}{V_n}$ and $\smuu{s}{V_n}$:
  \[\rmuu{s}{V_n} =\begin{cases} \frac{ns^2}{2} & \text{if } 0 < s\leq 1 \\ -\frac{n^2}{2}(s-1)^2+n(s-1)+\frac{n}{2} & \text{if } 1\leq s\leq 1+1/n \\ \frac{n+1}{2} & \text{if } s\geq 1+1/n\end{cases}\]
  \[\smuu{s}{V_n} =\begin{cases} n & \text{if } 0 < s < 1 \\ \frac{-\frac{n^2}{2}(s-1)^2+n(s-1)+\frac{n}{2}}{\frac{1}{2}s^2-(s-1)^2}& \text{if } 1\leq s < 1+1/n \\ \frac{n+1}{\frac{1}{2}s^2-(s-1)^2} & \text{if } 1+1/n\leq s < 2 \\  \frac{n+1}{2} & \text{if } s\geq 2.\end{cases}\]
 \end{example}

 \begin{example}
 The normalizing factors $\nfact{d}{s}$ can be easily visualized as areas in space in the same manner.  Indeed, since $k[[x_1,\ldots,x_d]]$ is a toric ring, we simply apply the construction above to calculate $\rmuu{s}{(x_1,\ldots,x_d)}$.  For instance, when $d=2$, we have the following picture: 
 \begin{center}

\begin{tikzpicture}[scale=1]
\path[fill=lightgray](2,0)--(0.5,0)--(0,0.5)--(0,2)--(2,2);
\foreach \x in {0,...,2}{
\foreach \y in {0,...,2}{
  \node[draw,circle,fill,inner sep=2] at (\x,\y){};
 }
}
\draw (2,0)--(0,0)--(0,2);
\draw (0.5,0)--(0,0.5);
\node [draw,circle,fill=red,inner sep=1.5] at (0,0.5) {};
\node [left,color=red] at (0,0.5) {$(0,s)$};
\node [draw,circle,fill=red,inner sep=1.5] at (0.5,0) {};
\node [below,color=red] at (0.5,0) {$(s,0)$};
\node [below] at (1,-0.5) {$0< s\leq 1$};
\node [below] at (1,-1) {$\nfact{2}{s}=\frac{1}{2}s^2$};
\end{tikzpicture}
\qquad
\begin{tikzpicture}[scale=1]
\path[fill=lightgray](2,0)--(1,0)--(1,0.5)--(0.5,1)--(0,1)--(0,2)--(2,2);
\foreach \x in {0,...,2}{
\foreach \y in {0,...,2}{
  \node[draw,circle,fill,inner sep=2] at (\x,\y){};
 }
}
\draw (2,0)--(0,0)--(0,2);
\draw (1,0)--(1,1)--(0,1);
\draw (1.5,0)--(0,1.5);
\node [draw,circle,fill=red,inner sep=1.5] at (0.5,1) {};
\node [above right,color=red] at (0.5,1) {$(s-1,1)$};
\node [draw,circle,fill=red,inner sep=1.5] at (1,0.5) {};
\node [right,color=red] at (1,0.5) {$(1,s-1)$};
\node [below] at (1,-0.5) {$1\leq s\leq 2$};
\node [below] at (1,-1) {$\nfact{2}{s}=\frac{1}{2}s^2-(s-1)^2$};
\end{tikzpicture}
\qquad
\begin{tikzpicture}[scale=1]
\path[fill=lightgray](2,0)--(1,0)--(1,1)--(0,1)--(0,2)--(2,2);
\foreach \x in {0,...,2}{
\foreach \y in {0,...,2}{
  \node[draw,circle,fill,inner sep=2] at (\x,\y){};
 }
}
\draw (2,0)--(0,0)--(0,2);
\draw (1,0)--(1,1)--(0,1);
\node [below] at (1,-0.5) {$s\geq 2$};
\node [below] at (1,-1) {$\nfact{2}{s}=1$};
\end{tikzpicture}
\end{center}
\end{example}
 \newpage
 \bibliographystyle{alpha}
 \bibliography{TaylorBibliography}
\end{document}